\documentclass[11pt]{article}
\usepackage{amssymb,amsmath, amsthm}
\usepackage{authblk}
\usepackage{enumerate}
\usepackage{hyperref}
\usepackage{cleveref}
\usepackage{empheq}
\crefformat{footnote}{#2\footnotemark[#1]#3}
\RequirePackage{algpseudocode,algorithm,algorithmicx}
\usepackage{enumitem}\usepackage{booktabs}
\setenumerate[0]{leftmargin=*}
\usepackage{diagbox}
\usepackage{caption}
\usepackage{graphicx}
\usepackage{float} 
\usepackage{subcaption}
\usepackage{color}
\usepackage{mathtools}

\usepackage{longtable}
\usepackage[sort&compress]{natbib}
\definecolor{myred}{rgb}{0.5,0,0}
\numberwithin{equation}{section}

\newcommand{\inclu}[0] {\ar@{^{(}->}}

\usepackage[margin=1.1in]{geometry}

\newtheorem{theorem}{Theorem}[section]
\newtheorem{proposition}[theorem]{Proposition}
\newtheorem{lemma}[theorem]{Lemma}

\newtheorem{remark}[theorem]{Remark}

\newtheorem{corollary}[theorem]{Corollary}

\newtheorem{example}{Example}[section]

\title{Improving Quasi-Newton Methods via Image and Projection Operators}

\author{
	Zhen-Yuan Ji\thanks{
		Email: \texttt{zyj@lsec.cc.ac.cn}
	}
}
\begin{document}
\maketitle	

\begin{abstract}
	Designing efficient quasi-Newton methods is an important problem in nonlinear optimization and the solution of systems of nonlinear equations. From the perspective of the matrix approximation process, this paper presents a unified framework for establishing the quadratic termination property that covers the Broyden family, the generalized PSB family, and Broyden's ``good''  method.
	 Based on this framework, we employ operators to map the correction direction $s_k$ in the quasi-Newton equation to a specific subspace, which ensures quadratic termination for these three classes of methods without relying on exact line searches. We derive the corresponding image and projection operators, analyze their improved properties in matrix approximation, and design practical algorithms accordingly. Preliminary numerical results show that the proposed operator-based methods yield significant improvements in the performance of the Davidon-Fletcher-Powell (DFP), Broyden-Fletcher-Goldfarb-Shanno (BFGS),  Powell-Symmetric-Broyden (PSB), limited-memory BFGS (L-BFGS) and Broyden's ``good''  methods (BGM).
\end{abstract}

\section{Introduction}\label{section1}
We study the process of matrix approximation:
\begin{align}
	B_{k+1} &= \mathcal{U}(B_k, s_k, y_k), \label{appm} \\
	\text{s.t.} \quad B_{k+1} s_k &= A_k s_k = y_k, \quad k = 0, 1, 2, \dots, \nonumber
\end{align}
where \( A_k, B_k : \mathbb{R}^n \to \mathbb{R}^n \) are linear operators, \( s_k \in \mathbb{R}^n \setminus \{0\} \), and \(\mathcal{U}(\cdot)\) denotes a generic update rule.  
This model is adopted in quasi-Newton methods, in which the sequence \( \{B_k\} \) approximates the Hessian or Jacobian matrix \( A_k \).

For example, consider the unconstrained smooth optimization problem
\begin{equation}
	\min_{x \in \mathbb{R}^n} f(x), \label{uo}\nonumber
\end{equation}
where \( f: \mathbb{R}^n \to \mathbb{R} \) is continuously differentiable. Given an initial point \( x_0 \) and an initial approximation matrix \( B_0 \), quasi-Newton methods generate iterates according to
\begin{equation}
	x_{k+1} = x_k - \alpha_k B_k^{-1} \nabla f(x_k), \label{standardQN1}\nonumber
\end{equation}
where \( \alpha_k > 0 \) is the step size, often determined by a line search procedure. The matrix \( B_k \) is updated to satisfy the quasi-Newton equation
\begin{equation}
	B_{k+1} s_k = y_k, \label{standardQN2}
\end{equation}
with
\[
s_k = x_{k+1} - x_k, \quad y_k = \nabla f(x_{k+1}) - \nabla f(x_k).
\]

Well-known quasi-Newton update formulas for solving optimization problems include the Davidon-Fletcher-Powell (DFP) method\textsuperscript{\cite{Fletcher1963ARC, Davidon1991VariableMM}}, the Broyden-Fletcher-Goldfarb-Shanno (BFGS) method\textsuperscript{\cite{Broyden1970TheCO, Goldfarb1970AFO, Fletcher1970ANA, Shanno1970ConditioningOQ}}, and the symmetric rank-one (SR1) method\textsuperscript{\cite{Davidon1991VariableMM, Conn1991ConvergenceOQ}}. These methods all belong to the Broyden family of quasi-Newton methods\textsuperscript{\cite{Broyden1967QuasiNewtonMA}}. There are many studies on the global convergence of the Broyden family methods \textsuperscript{\cite{Byrd1987GlobalCO,dai2002convergence,Li2001OnTG,Mor1974OnTG,Dai2013APE,Mascarenhas2004TheBM}}.  

The Broyden family update formula is given by:
\begin{align}\label{Broyden2}
	B_{k+1}^{\theta} =  B_k + \frac{y_k y_k^\top}{s_k^\top y_k} - \frac{B_k s_k s_k^\top B_k}{s_k^\top B_k s_k} + \theta \, \omega_k \omega_k^\top,
\end{align}
where
\begin{equation}\label{omega}
	\omega_k = \sqrt{s_k^\top B_k s_k} \left(\frac{y_k}{s_k^\top y_k} - \frac{B_k s_k}{s_k^\top B_k s_k} \right).\nonumber
\end{equation}
In particular,  when \(\theta = 1\), the formula \eqref{Broyden2} reduces to the DFP update; when \(\theta = 0\), it reduces to the BFGS update.

Limited-memory BFGS (L-BFGS)\textsuperscript{\cite{liu1989limited}} is a widely adopted quasi-Newton method tailored for large-scale unconstrained optimization problems. Instead of explicitly storing the full inverse Hessian approximation, L-BFGS maintains only a limited number of the most recent secant pairs to implicitly represent the inverse Hessian. This makes it particularly suitable for high-dimensional problems where storing a full Hessian approximation would be prohibitively expensive.

The Powell-Symmetric-Broyden (PSB) method\textsuperscript{\cite{powell1970new, powell1975convergence}} is  an important quasi-Newton method. It possesses the least-change property under the Frobenius norm, which endows it with useful theoretical properties for analyzing quasi-Newton methods.

The generalized PSB update formula is given by:
\begin{align}\label{GPSB2}
	B_{k+1} &= B_k + \frac{(y_k - B_k s_k) s_k^T M^{-2} + M^{-2} s_k (y_k - B_k s_k)^T}{s_k^T M^{-2} s_k} \nonumber \\
	&\quad + \frac{s_k^T (y_k - B_k s_k)}{(s_k^T M^{-2} s_k)^2} M^{-2} s_k s_k^T M^{-2}.
\end{align}
where \( M \) is a symmetric nonsingular matrix. When \( M \) is the identity matrix, the update reduces to the standard PSB update.

Broyden's ``good''  method (BGM)\textsuperscript{\cite{broyden1965class}} is a widely used quasi-Newton method for solving nonlinear systems
\[
F(x) = 0, \quad F: \mathbb{R}^n \to \mathbb{R}^n,
\]
particularly when explicit computation of the Jacobian $J(x) = \nabla F(x)$ is costly or infeasible.  At each iteration, BGM chooses $B_{k+1}$ so that the quasi-Newton equation
\[
B_{k+1}s_k = y_k, \quad s_k = x_{k+1} - x_k, \quad y_k = F(x_{k+1}) - F(x_k)
\]
is satisfied, while applying the minimal rank-one correction in the Frobenius norm to $B_k$ without enforcing symmetry.  The resulting update is
\begin{equation}
	B_{k+1} = B_k + \frac{(y_k - B_k s_k)s_k^\top}{s_k^\top s_k}.\label{B1}
\end{equation}
The next iterate is then computed via
\begin{equation}\label{itBGM}
	x_{k+1} = x_k - \alpha_kB_k^{-1} F(x_k).\nonumber
\end{equation}

Although in the general case the matrix $A_k$ depends on the current iterate $x_k$, we consider here a simplified setting where $A$ is fixed. 
This setting allows us to isolate and analyze the fundamental behavior of the update rule~\eqref{appm}.

\textbf{Motivation.} It can be observed that when the correction direction $s_k$ lies in or close to the null space of $B_k - A$, the quasi-Newton update~\eqref{appm} becomes ineffective. 
A simple example demonstrates that a poor initial point or an inaccurate approximation $B_0$ can cause the directions $s_k$ generated by DFP and BFGS to become trapped in $\ker(B_k - A)$, leading to stagnation and degraded numerical performance.

\begin{example}
	Consider minimizing the quadratic function
	\[
	f(x, y) = \frac{1}{2}(x^2 + y^2),
	\]
	with initial point
	\[
	x_0 = \begin{bmatrix} \cos \theta \\ \sin \theta \end{bmatrix}, \quad \theta = 89^\circ,
	\]
	and initial approximate matrix
	\[
	B_0 = \begin{bmatrix} 1 & 0 \\ 0 & 10^6 \end{bmatrix}.
	\]
	The stopping criterion is
	\[
	\| \nabla f(x_k) \|_2 \le 10^{-6} \| \nabla f(x_0) \|_2.
	\]
	The iteration scheme is given by
	\[
	x_{k+1} = x_k - B_k^{-1} \nabla f(x_k).
	\]
	
	Under these conditions:
	\begin{itemize}
		\item The DFP method terminates after 37,554 iterations, with the average angle between the correction directions \( s_k \) and the null space \(\ker(B_k - A)\) being approximately \(0.6939^\circ\).
		\item The BFGS method terminates after only 16 iterations, though the early steps contribute little to effectively enforcing the secant condition. The angles between \( s_k \) and \(\ker(B_k - A)\) at each iteration are listed below.
	\end{itemize}
	
	\begin{table}[H]
		\centering\small
		\begin{tabular}{|c|c|c|c|c|c|c|}
			\hline
			\textbf{Iteration} & 0 & 1 & 2 & 3 & 4 & 5 \\ \hline
			\textbf{Angle (degrees)} & 0.0038 & 0.0044 & 0.0063 & 0.0083 & 0.9935 & 4.9044 \\ \hline
			\textbf{Iteration} & 6 & 7 & 8 & 9 & 10 & 11 \\ \hline
			\textbf{Angle (degrees)} & 7.5018 & 18.5168 & 30.8080 & 26.8430 & 34.0865 & 39.0512 \\ \hline
			\textbf{Iteration} & 12 & 13 & 14 & 15 & 16 & \\ \hline
			\textbf{Angle (degrees)} & 41.0937 & 43.4249 & 44.2037 & 44.6014 & 44.7988 & \\ \hline
		\end{tabular}
		\caption{Angles between the BFGS search directions $s_k$ and $\ker(B_k - A)$ over iterations}
	\end{table}
\end{example}

\textbf{Our contribution.} The main contributions of this paper are summarized as follows.  
\begin{itemize}
	\item From the perspective of matrix approximation, we provide a sufficient condition for the quadratic termination property of quasi-Newton methods. Specifically, the correction direction \( s_k \) lies in the orthogonal complement of \(\ker(B_k - A)\) under a certain inner product. We further derive these corresponding inner products for the Broyden family, the generalized PSB family, and BGM.  
	
	\item We propose a technique to improve matrix approximation by mapping the correction direction \( s_k \) onto the orthogonal complement of \(\ker(B_k - A)\) via a linear operator. By representing this orthogonal complement as the image of an appropriate operator, we construct image operators tailored to the Broyden family, the generalized PSB family, and BGM. Then, we analyze their effectiveness in enhancing matrix approximation accuracy. Moreover, we introduce projection operators and demonstrate their improved matrix approximation properties.  
	
	\item Based on the image and projection operators, we design practical algorithms. Preliminary numerical experiments demonstrate significant improvements for DFP, BFGS, PSB, L-BFGS, and BGM using the proposed methods.
\end{itemize}

\textbf{Content.}  
This paper is organized as follows:  
Section \ref{section2} introduces a unified quadratic termination framework for the Broyden family, the generalized PSB family, and BGM, followed by the development of corresponding image and projection operator approaches.  
Section \ref{section3} discusses the improved matrix approximation properties of DFP, BFGS, the generalized PSB family, and BGM when applying these operators.  
Section \ref{section4} presents algorithm designs based on the proposed operators.  
Section \ref{section5} presents preliminary numerical experiments that compare standard DFP, BFGS, PSB, and L-BFGS with their operator-based variants on quadratic functions.  
Additionally, it compares Newton’s method, standard BGM, and BGM with the projection operator on two nonlinear systems.
Section \ref{c:6} concludes with a discussion and summary of the results.

\section{Image Operator and Projection Operator Approaches}\label{section2}
In this section, we consider the Broyden family and the generalized PSB methods under the assumption that both \( A \) and \( B_k \) are symmetric.  
This assumption does not hold for BGM, where symmetry is not required.

Define the inner product
\[
\langle a, b \rangle_W = a^\top W b, \quad \forall a, b \in \mathbb{R}^n,
\]
where \( W \in \mathbb{R}^{n \times n} \) is a symmetric positive definite matrix.  
The norm induced by the \( W \)-inner product is defined as
\[
\| a \|_W = \sqrt{\langle a, a \rangle_W} = \sqrt{a^\top W a}, \quad \forall a \in \mathbb{R}^n.
\]

The following proposition provides a unified framework for several update schemes by demonstrating that, with respect to a suitable inner product \( \langle \cdot, \cdot \rangle_W \), the matrix approximation process terminates in at most \( n \) steps.

\begin{proposition}\label{p1}
	Consider the matrix approximation process~\eqref{appm}. Suppose that, at each step, a nonzero correction direction \( s_k \in \ker(B_k - A)^{\perp_W} \) is chosen with respect to the inner product \( \langle a, b \rangle_W = a^\top W b \).  
	Then the sequence \( \{B_k\} \) terminates at \( B_n = A \) in at most \( n \) steps.  
	
	Specifically, termination is guaranteed under the following choices of \( W \) for different update rules:
	\begin{itemize}
		\item \textbf{Broyden family}: \( W = A \) or \( W = B_k \).
		\item \textbf{Generalized PSB family}: \( W = M^{-2} \), where \( M \) is the symmetric nonsingular matrix associated with the update.
		\item \textbf{BGM}: \( W = I \) (the identity matrix).
	\end{itemize}
	
	Here, \( \ker(B_k - A)^{\perp_W} \) denotes the orthogonal complement of \( \ker(B_k - A) \) with respect to the inner product \( \langle \cdot, \cdot \rangle_W \).
\end{proposition}

\begin{proof}
	It suffices to show that if \( B_k \neq A \), then
	\[
	\ker(B_{k+1} - A) \supsetneq \ker(B_k - A),
	\]
	which implies the kernel strictly increases at each iteration.
	
	Since \( s_k \in \ker(B_k - A)^{\perp_W} \setminus \{0\} \), it follows that \( s_k \notin \ker(B_k - A) \). Moreover, the quasi-Newton equation in~\eqref{appm} ensures \( s_k \in \ker(B_{k+1} - A) \), hence
	\[
	\ker(B_{k+1} - A) \neq \ker(B_k - A).
	\]
	
	Let \( \xi \in \ker(B_k - A) \). To complete the proof, we verify \( \xi \in \ker(B_{k+1} - A) \), which reduces to showing that the update term \( (B_{k+1} - B_k) \xi = 0 \) for each update formula.
	
	\textbf{For the Broyden family update~\eqref{Broyden2}:} Taking \( W = A \), we have
	\[
	y_k^T \xi = \langle s_k, \xi \rangle_A = 0, \quad s_k^T B_k \xi = s_k^T (B_k - A) \xi + s_k^T A \xi = 0.
	\]
	Taking \( W = B_k \), 
	\[
	s_k^T B_k \xi = \langle s_k, \xi \rangle_{B_k} = 0, \quad y_k^T \xi = s_k^T A \xi = -s_k^T (B_k - A) \xi + s_k^T B_k \xi = 0.
	\]
	Thus, \( (B_{k+1} - B_k) \xi = 0 \) in both cases.
	
	\textbf{For the generalized PSB family update~\eqref{GPSB2}:} Taking \( W = M^{-2} \), we have
	\[
	s_k^T M^{-2} \xi = \langle s_k, \xi \rangle_{M^{-2}} = 0, \quad (y_k - B_k s_k)^T \xi = s_k^T (A - B_k)^T \xi = s_k^T (A - B_k) \xi= 0.
	\]
	Hence \( (B_{k+1} - B_k) \xi = 0 \).
	
	\textbf{For the BGM update~\eqref{B1}:} Taking \( W = I \), we have
	\[
	s_k^T \xi = \langle s_k, \xi \rangle_I = 0,
	\]
	yielding \( (B_{k+1} - B_k) \xi = 0 \).
	
	Therefore, \( \ker(B_{k+1} - A) \supsetneq \ker(B_k - A) \) whenever \( B_k \neq A \), which concludes the proof.
\end{proof}

Our core idea is to map the correction direction \( s_k \) so that it lies in the orthogonal complement of \( \ker(B_k - A) \) with respect to the inner product \( \langle \cdot, \cdot \rangle_W \). By Proposition~\ref{p1}, an algorithm designed in this way possesses the quadratic termination property without relying on exact line searches, that is, \( B_n = A \) in at most \( n \) steps. 

We next introduce two operators to achieve this goal, namely the image operator and projection operator.

\subsection{Image Operator Approaches}\label{Image operator approaches}
In this subsection, we construct a linear operator that maps \( s_k \) onto \( \ker(B_k - A)^{\perp_W} \).
\begin{proposition}\label{p2}
Let $E\in\mathbb{R}^{n\times n}$ and let $W\in\mathbb{R}^{n\times n}$ be symmetric positive definite.  
With respect to the weighted inner product
\[
\langle a , b\rangle_W = a^{T}W\,b ,
\qquad for\  \forall a,b\in\mathbb{R}^n,
\]
the $W$-orthogonal complement of $\ker E$ is
\[
\boxed{\;(\ker E)^{\perp_W} \;=\; \operatorname{Im}\bigl(W^{-1}E^{T}\bigr).\; } 
\]
\end{proposition}
\begin{proof}  
By definition, the set \((\ker E)^{\perp_W}\) is given by
\[
(\ker E)^{\perp_W} = \bigl\{ z \in \mathbb{R}^n \mid z^{T} W x = 0 \quad \forall x \in \ker E \bigr\}.
\]
We will prove the equivalence
\[
z \in (\ker E)^{\perp_W} \quad \Longleftrightarrow \quad \exists y \in \mathbb{R}^n \text{ such that } z = W^{-1} E^{T} y.
\]

\medskip  
\noindent  
\textbf{(\(\Rightarrow\))}  
Assume \(z \in (\ker E)^{\perp_W}\).  
Consider the linear system
\[
E^{T} y = W z.
\]  
By Farkas’ lemma, exactly one of the following systems is feasible:
\[
\text{either}\quad E^{T} y = W z \quad \text{or} \quad \exists x \neq 0 : E x = 0 \text{ and } z^{T} W x \neq 0.
\]
Suppose, for contradiction, that the first system is not solvable.  
Then there exists \(x \neq 0\) such that \(E x = 0\) and \(z^{T} W x \neq 0\).  
However, since \(x \in \ker E\) and \(z \in (\ker E)^{\perp_W}\), it follows that
\[
z^{T} W x = 0,
\]
a contradiction.  
Hence, the system \(E^{T} y = W z\) admits a solution \(y\).  
Multiplying both sides by \(W^{-1}\), we obtain
\[
z = W^{-1} E^{T} y,
\]
which shows that \(z \in \operatorname{Im}(W^{-1} E^{T})\).

\medskip  
\noindent  
\textbf{(\(\Leftarrow\))}  
Conversely, let \(z = W^{-1} E^{T} y\) for some \(y \in \mathbb{R}^n\).  
For any \(x \in \ker E\), we have
\[
\langle z, x \rangle_W = z^{T} W x = (W^{-1} E^{T} y)^{T} W x = y^{T} E x = 0,
\]
since \(E x = 0\).  
Therefore, \(z \in (\ker E)^{\perp_W}\).

\medskip  
Both directions established, we conclude
\[
(\ker E)^{\perp_W} = \operatorname{Im}(W^{-1} E^{T}),
\]
as claimed. 
\end{proof}

\begin{proposition}\label{p2.3}
	In step \(k\) of the matrix approximation problem~\eqref{appm}, replace \(s_k\) by 
	\[
	\tilde{s}_k = W^{-1}(B_k - A)^T s_k,
	\]
	and \(y_k\) by \(\tilde{y}_k = A \tilde{s}_k\), for \(k = 0, 1, 2, \dots\). Assume that \(\tilde{s}_k \neq 0\) for all \(k\). Then, for the matrix update formulas mentioned above, there exists a corresponding symmetric positive definite \(W\) such that the problem~\eqref{appm} terminates in at most \(n\) steps, i.e., \(B_n = A\).
\end{proposition}

\begin{proof}
	Since \(\tilde{s}_0 \neq 0\), by the quasi-Newton equation we have \(\tilde{s}_0 \in \ker(B_1 - A)\), hence \(\dim \ker(B_1 - A) \geq 1\). By Proposition~\ref{p2}, it holds that \(\tilde{s}_k \in \ker(B_k - A)^{\perp_W}\) for all \(k \ge 1\). Using the assumption \(\tilde{s}_k \neq 0\) and Proposition~\ref{p1}, it follows that
	\[
	\ker(B_{k+1} - A) \supsetneq \ker(B_k - A), \quad \forall k \ge 1.
	\]
	Therefore, the problem~\eqref{appm} terminates in at most \(n\) steps, i.e., \(B_n = A\).
\end{proof}

\begin{remark}\label{remark1}
	From the above proposition, we observe that the corresponding linear operators for the Broyden family, the generalized PSB family, and the BGM are respectively
	\[
	B_k^{-1}(B_k - A) \text{ or } A^{-1}(B_k - A), \quad M^2(B_k - A), \quad \text{and} \quad (B_k - A)^T.
	\]
	Since these operators arise from representations of image spaces, we refer to them as \emph{image operators}.
\end{remark}

Note that the operators \( A^{-1}(B_k - A) \) and \( (B_k - A)^T \) are generally impractical to compute because \( A^{-1} \) and \( A^T \) are typically unavailable or unreliable to approximate. Therefore, only two of these image operators are computationally viable: \( B_k^{-1}(B_k - A) \), used in the Broyden family methods, and \( M^2(B_k - A) \), used in the generalized PSB family methods. We will design algorithms based on these operators in Section~\ref{section4}.

\subsection{Projection operator approaches}\label{Projection operator approaches}
Given the structure of the space \( \ker(B_k - A)^{\perp_W} \), a natural approach is to project the correction direction \( s_k \) onto this space with respect to the inner product \( \langle \cdot, \cdot \rangle_W \). We denote this projection operator by \( \operatorname{Proj}_{\ker(B_k - A)^{\perp_W}}^W \).

However, direct computation of \( \operatorname{Proj}_{\ker(B_k - A)^{\perp_W}}^W s_k \) is generally infeasible; it requires knowledge of the Hessian or Jacobian matrix \( A \), which is typically unavailable.
To guarantee finite termination, Proposition~\ref{p1} requires that \( s_k \in \ker(B_k - A)^{\perp_W} \). In fact, it suffices to choose \( s_k \) orthogonal, with respect to the \( W \)-inner product, to the subspace spanned by the previously corrected directions, i.e.,
\[
s_k \in \operatorname{span}\{ s_0, s_1, \ldots, s_{k-1} \}^{\perp_W}.
\]

\begin{proposition}\label{p3}
	Consider the matrix approximation process~\eqref{appm}. Let \( W \) be a symmetric positive definite matrix defining the inner product \( \langle \cdot, \cdot \rangle_W \). Suppose the sequence \( \{s_0, s_1, \ldots, s_{n-1}\} \) satisfies
	\begin{equation}\label{orthgonal}
		s_k \in \operatorname{span}\{s_0, s_1, \ldots, s_{k-1}\}^{\perp_W}, \quad k = 1, \ldots, n-1.
	\end{equation}
	
	Then, for the update formulas discussed above, and with \( W \) chosen accordingly, the matrix approximation problem~\eqref{appm} terminates in at most \( n \) steps, i.e., \( B_n = A \). More specifically, \( W \) is chosen as follows depending on the update rule:
	\begin{itemize}
		\item \textbf{Broyden family:} \( W = A \) or \( W = B_k \).
		\item \textbf{Generalized PSB family:} \( W = M^{-2} \), where \( M \) is the symmetric nonsingular matrix defining the update.
		\item \textbf{BGM:} \( W = I \) (the identity matrix).
	\end{itemize}
\end{proposition}

\begin{proof}  
We prove by induction that for all \( i \ge 1 \), the relation \( B_i s_j = y_j \) holds for all \( j = 0, 1, \ldots, i-1 \).

For \( i = 1 \), we have \( B_1 s_0 = y_0 \) by construction of the update. 

Suppose the statement holds for \( i = k \), that is,
\[
B_k s_j = y_j, \quad \forall j = 0, 1, \ldots, k-1.
\]
Assume that for some \(k \ge 1\),
\[
B_k s_j = y_j, \quad \forall j = 0, \ldots, k-1.
\]
We need to prove the property for \( i = k+1 \). By the update definition,
\(B_{k+1} s_k = y_k. \)
It remains to show
\[
B_{k+1} s_j = y_j, \quad \forall j = 0, \ldots, k-1,
\]
or equivalently,
\[
(B_{k+1} - B_k) s_j = 0, \quad \forall j = 0, \ldots, k-1.
\]

\textbf{For the Broyden family update~\eqref{Broyden2}:} Taking \( W = A \),  we have
\[
y_k^T s_j = \langle s_k, s_j \rangle_A = 0, \quad s_k^T B_k s_j = s_k^T y_j = \langle s_k, s_j \rangle_A = 0.
\]
Taking instead \( W = B_k \), 
\[
s_k^T B_k s_j = \langle s_k, s_j \rangle_{B_k} = 0, \quad y_k^T s_j = s_k^T A s_j = -s_k^T (B_k - A) s_j + s_k^T B_k s_j = 0.
\]
Thus, in both choices of \( W \), we obtain \( (B_{k+1} - B_k) s_j = 0 \).

\textbf{For the generalized PSB family update~\eqref{GPSB2}:} Taking \( W = M^{-2} \), 
\[
s_k^T M^{-2} s_j = \langle s_k, s_j \rangle_{M^{-2}} = 0, \quad (y_k - B_k s_k)^T s_j = s_k^T (A - B_k)^T s_j=s_k^T (A - B_k)s_j = 0.
\]
It follows that \( (B_{k+1} - B_k) s_j = 0 \).

\textbf{For the BGM update~\eqref{B1}:}  Taking \( W = I \), 
\[
s_k^T s_j = \langle s_k, s_j \rangle_I = 0,
\]
which  implies \( (B_{k+1} - B_k) s_j = 0 \).

By induction, it follows that
\[
B_i s_j = y_j, \quad \forall i = 1, \ldots, n, \quad j = 0, \ldots, i-1.
\]
From the orthogonality condition~\eqref{orthgonal} and the assumption \( s_k \neq 0 \), the set \( \{ s_0, s_1, \ldots, s_{n-1} \} \) is linearly independent. Consequently, the matrix approximation problem~\eqref{appm} terminates in at most \( n \) steps, i.e., \( B_n = A \).
\end{proof}

Proposition~\ref{p3} motivates a stepwise orthogonalization of the sequence \( \{s_i\} \) generated by the quasi-Newton method. This procedure can be viewed as a weighted version of the classical Gram-Schmidt orthogonalization with respect to the inner product \( \langle \cdot, \cdot \rangle_W \).

We start by setting \(\tilde{s}_0 = s_0\), and define
\[
\tilde{s}_1 = s_1 + \alpha_{10} \tilde{s}_0,
\]
where the coefficient \(\alpha_{10}\) is determined by enforcing the orthogonality condition
\[
\langle \tilde{s}_0, \tilde{s}_1 \rangle_W = 0,
\]
which yields
\[
\alpha_{10} = \frac{-\langle s_1, \tilde{s}_0 \rangle_W}{\langle \tilde{s}_0, \tilde{s}_0 \rangle_W}.
\]

More generally, for each \(k = 1, \ldots, n-1\), we define
\begin{equation}\label{eq2.2}
	\tilde{s}_k = s_k + \sum_{j=0}^{k-1} \alpha_{k,j} \tilde{s}_j,
\end{equation}
where the coefficients \(\alpha_{k,j}\) satisfy the orthogonality conditions
\begin{equation}\label{or1}
	\langle \tilde{s}_k, \tilde{s}_i \rangle_W = 0, \quad i = 0, 1, \ldots, k-1.
\end{equation}
These conditions lead to the explicit formula
\begin{equation}
\alpha_{k,j} = \frac{-\langle s_k, \tilde{s}_j \rangle_W}{\langle \tilde{s}_j, \tilde{s}_j \rangle_W}.\nonumber
\end{equation}
\begin{remark}
	For improved numerical stability, one may adopt the Modified Gram-Schmidt method, which orthogonalizes a vector by sequentially subtracting its projections onto all previously orthogonalized vectors. Specifically, at the \( k \)-th step, initialize \( \tilde{s}_k = s_k \), and then for \( j = 0, 1, \dots, k-1 \), update
	\[
	\tilde{s}_k := \tilde{s}_k - \frac{ \langle \tilde{s}_k, \tilde{s}_j \rangle_W }{ \langle \tilde{s}_j, \tilde{s}_j \rangle_W } \tilde{s}_j.
	\]
	This approach is known to be numerically more stable than the classical Gram-Schmidt process.
\end{remark}

\begin{proposition}\label{p4}
	If the original set of vectors \( \{s_0, s_1, \ldots, s_{n-1}\} \) is linearly independent, then the orthogonalization process~\eqref{eq2.2} produces a sequence \( \{\tilde{s}_0, \tilde{s}_1, \ldots, \tilde{s}_{n-1}\} \) that is also linearly independent. 
	If the matrix update is performed using the orthogonalized sequence \( \{\tilde{s}_0, \tilde{s}_1, \ldots, \tilde{s}_{n-1}\} \), the matrix approximation process~\eqref{appm} terminates in at most \( n \) steps. 
	Moreover, we have
	\begin{equation}\label{orthogonalizationbelong}
		\operatorname{span}\{s_0, \ldots, s_{k-1}\} \subseteq \ker(B_k - A), \quad k = 0, 1, \ldots, n.
	\end{equation}
\end{proposition}

\begin{proof}
	Since \( \{\tilde{s}_0, \tilde{s}_1, \ldots, \tilde{s}_{n-1}\} \) is the orthogonalized sequence of \( \{s_0, s_1, \ldots, s_{n-1}\} \), it follows that
	\begin{equation}\label{eq2.7.3}
		\operatorname{span}\{\tilde{s}_0, \tilde{s}_1, \ldots, \tilde{s}_{k-1}\}
		= \operatorname{span}\{s_0, s_1, \ldots, s_{k-1}\}, 
		\quad k = 1, 2, \ldots, n,
	\end{equation}
	which  ensures that the linear independence of \( \{s_0, \ldots, s_{n-1}\} \) implies the linear independence of \( \{\tilde{s}_0, \ldots, \tilde{s}_{n-1}\} \).
	
	By Proposition~\ref{p3}, the matrix approximation process terminates in at most \( n \) steps when using the orthogonalized sequence.
	
	Furthermore, from the proof of Proposition~\ref{p3}, we have
	\begin{equation}\label{eq2.7.2}
		\operatorname{span}\{\tilde{s}_0, \tilde{s}_1, \ldots, \tilde{s}_{k-1}\} \subseteq \ker(B_k - A), \quad k = 0, 1, \ldots, n.
	\end{equation}
	Hence, \eqref{orthogonalizationbelong} follows immediately from \eqref{eq2.7.3} and \eqref{eq2.7.2}.
\end{proof}

Proposition~\ref{p4} implies that the orthogonalization procedure projects \( s_k \) onto a subspace of \( \ker(B_k - A) \).  
More generally, one may project \( s_k \) onto the subspace spanned by the most recent \( d \) orthogonalized correction directions, i.e., \( \operatorname{span}\{\tilde{s}_{k-d}, \ldots, \tilde{s}_{k-1}\} \), by applying stepwise orthogonalization:
\begin{equation}\label{or2}
	\tilde{s}_k = s_k + \sum_{j = \max\{0, k - d\}}^{k - 1} \alpha_{k,j} \tilde{s}_j,\quad k\ge1,
\end{equation}
where \( 0 \le d \le n - 1 \) is a fixed integer parameter. When \( d = 0 \), this reduces to the standard scheme.

According to Proposition~\ref{p3}, the coefficients \( \alpha_{k,j} \) in \eqref{or1} and \eqref{or2} are calculated differently depending on the update formula:
\begin{itemize}
	\item \textbf{For the Broyden Family Methods:}
	\begin{equation}\label{O1}
		\alpha_{k,j} = \frac{- s_k^\top \tilde{y}_j}{\tilde{s}_j^\top \tilde{y}_j}.
	\end{equation}
	\item \textbf{For the Generalized PSB Methods:}
	\begin{equation}\label{O2}
		\alpha_{k,j} = \frac{- s_k^\top M^{-2} \tilde{s}_j}{\tilde{s}_j^\top M^{-2} \tilde{s}_j}.
	\end{equation}
	\item \textbf{For the BGM Methods:}
	\begin{equation}\label{O3}
		\alpha_{k,j} = \frac{- s_k^\top \tilde{s}_j}{\tilde{s}_j^\top \tilde{s}_j}.
	\end{equation}
\end{itemize}

\section{Improved Matrix Approximation Properties}\label{section3}
This section investigates the matrix approximation properties of the two operator-based approaches introduced in Section~\ref{section2}, highlighting their improvements over standard methods. In particular, we analyze the image operators for the Broyden family, the generalized PSB family, and BGM as mentioned in Remark~\ref{remark1}, as well as the impact of the projection operator \( \operatorname{Proj}_{\ker(B_k - A)}^W \) and the stepwise orthogonalization procedures~\eqref{eq2.2} and \eqref{or2} on the approximation accuracy.

\subsection{Proposition on Matrix Approximation for the Image Operator}
We begin by analyzing the operator corresponding to the generalized PSB family. According to Remark~\ref{remark1}, it can be interpreted as updating the matrix by applying the following quasi-Newton equation:
\begin{equation}\label{Qk}
	B_{k+1} W_k s_k = A W_k s_k,
\end{equation}
where the associated operator is defined as
\begin{equation}\label{Pk}
	W_k = M^2 (B_k - A).\nonumber
\end{equation}

To compare the quality of matrix approximations, we introduce a classical matrix error estimate for the generalized PSB family~\eqref{GPSB2}.  Before that, we present the following lemma.

\begin{lemma}\label{lem3.1}
	Let \(C, D \in \mathbb{R}^{n \times n}\) be symmetric matrices, and \(s \in \mathbb{R}^n \setminus \{0\}\). If
	\begin{equation}\label{3.1.3}
		D = \left(I - \frac{ss^\top}{\|s\|_2^2}\right) C \left(I - \frac{ss^\top}{\|s\|_2^2}\right),\nonumber
	\end{equation}
	then it holds that
	\begin{equation}\label{3.1.4}
		\|D\|_F^2 \le \|C\|_F^2 - \frac{\|C s\|_2^2}{\|s\|_2^2},\nonumber
	\end{equation}
where \(\|\cdot\|_F\) denotes the Frobenius norm defined by
\[
\|C\|_F = \sqrt{\operatorname{tr}(C^\top C)},
\]
and \(\operatorname{tr}(\cdot)\) is the trace operator defined as the sum of the diagonal elements of a square matrix, i.e.,
\[
\operatorname{tr}(A) = \sum_{i=1}^n A_{ii}, \quad \text{for } A \in \mathbb{R}^{n \times n}.
\]
\end{lemma}

\begin{proof}
	If \(Cs = 0\), then \(C = D\), so the inequality trivially holds.
	
	Now assume \(Cs \neq 0\). Using the cyclic property of the trace and the symmetry of \(C\), direct expansion yields
	\[
	\|D\|_F^2 = \operatorname{tr}\left(D^\top D\right) 
	= \operatorname{tr}\left(C^\top C\right) - 2 \frac{s^\top C^2 s}{s^\top s} + \frac{(s^\top C s)^2}{(s^\top s)^2}.
	\]	
	Furthermore, note that
	\[
	\quad |s^\top C s| \le \|s\|_2 \|C s\|_2
	\]
	by Cauchy-Schwarz inequality. Hence, we have
	\[
	\frac{(s^\top C s)^2}{(s^\top s)^2} \le \frac{\|C s\|_2^2}{\|s\|_2^2}.
	\]
	Putting these together,
	\[
	\|D\|_F^2 \le \|C\|_F^2 - 2 \frac{\|C s\|_2^2}{\|s\|_2^2} + \frac{\|C s\|_2^2}{\|s\|_2^2} = \|C\|_F^2 - \frac{\|C s\|_2^2}{\|s\|_2^2}.
	\]
	This completes the proof.
\end{proof}

\begin{proposition}
	Let \(A, B, M \in \mathbb{R}^{n \times n}\) be symmetric matrices and \(M\) is nonsingular. Denote the weighted Frobenius norm as
	\[
	\|X\|_{M,F} := \|M X M\|_F.
	\]
Then for any vector \(s \ne 0\), the generalized PSB update~\eqref{GPSB2} with $As=y$ has the following approximation property:
	\begin{equation}\label{3.1.6}
		\|PSB_M(B, A, s) - A\|_{M,F}^2 \le \|B - A\|_{M,F}^2 - \frac{\|M(B - A)s\|_2^2}{\|M^{-1}s\|_2^2}.
	\end{equation}
\end{proposition}

\begin{proof}
	Note that the generalized PSB update~\eqref{GPSB2} satisfies the identity:
	\[
	M(PSB_M(B, A, s) - A)M = \left(I - \frac{M^{-1}s s^\top M^{-1}}{s^\top M^{-2}s}\right) M(B - A)M \left(I - \frac{M^{-1}s s^\top M^{-1}}{s^\top M^{-2}s} \right).
	\]
By Lemma~\ref{lem3.1}, applying it with the substitutions
\[
D \leftarrow M(PSB_M(B, A, s) - A)M, \quad C \leftarrow M(B - A)M, \quad s \leftarrow M^{-1}s,
\]
and using the definition of the weighted Frobenius norm, we obtain inequality~\eqref{3.1.6}.
\end{proof}

Inequality~\eqref{3.1.6} gives an lower bound on the reduction in the matrix approximation error for the generalized PSB family:
\[
\frac{\|M(B - A)s\|_2^2}{\|M^{-1}s\|_2^2}.
\]
We next show that the application of the operator \( W_k \) yields a better lower bound.
\begin{lemma}\label{4.3}
	Let \( B \in \mathbb{R}^{n \times n} \) be a symmetric matrix, and define
	\[
	l(u) = \frac{u^\top B^2 u}{u^\top u}.
	\]
	Then for any vector \( u \in \mathbb{R}^n \) such that \( Bu \ne 0 \), we have
	\begin{equation}\label{lu2}
		l(Bu) \ge l(u),
	\end{equation}
	with equality if and only if \( u \) is an eigenvector of \( B \).
\end{lemma}

\begin{proof}
	Since \( B \in \mathbb{R}^{n \times n} \) is symmetric, it admits an orthogonal eigendecomposition. Let \[ u = \sum_{i=1}^{n} u^i, \] where each \( u^i \) denotes the component of \( u \) along the \( i \)-th eigenvector, satisfying
	\[
	B u^i = \lambda_i u^i, \quad \langle u^i, u^j \rangle = 0 \quad \text{for } i \ne j.
	\]
	
	Then proving~\eqref{lu2} is equivalent to verifying the inequality
	\[
	\frac{\sum_{i=1}^{n} \lambda_i^4 \|u^i\|^2}{\sum_{i=1}^{n} \lambda_i^2 \|u^i\|^2} \ge \frac{\sum_{i=1}^{n} \lambda_i^2 \|u^i\|^2}{\sum_{i=1}^{n} \|u^i\|^2},
	\]
	which is equivalent to
	\begin{equation}\label{pai1}
		\left( \sum_{i=1}^{n} \lambda_i^4 \|u^i\|^2 \right) \left( \sum_{i=1}^{n} \|u^i\|^2 \right) \ge \left( \sum_{i=1}^{n} \lambda_i^2 \|u^i\|^2 \right)^2.
	\end{equation}
	Expanding both sides of~\eqref{pai1} and simplifying yields
	\begin{align}\label{pai3}
		\sum_{1 \le i < j \le n} \left( \lambda_i^4 + \lambda_j^4 \right) \|u^i\|^2 \|u^j\|^2
		\ge \sum_{1 \le i < j \le n} 2 \lambda_i^2 \lambda_j^2 \|u^i\|^2 \|u^j\|^2.
	\end{align}
	The inequality~\eqref{pai3} follows from the basic inequality
	\[
	\lambda_i^4 + \lambda_j^4 \ge 2 \lambda_i^2 \lambda_j^2,
	\]
	with equality if and only if \( \lambda_i = \lambda_j \) for all \( i \ne j \), which implies that \( u \) is an eigenvector of \( B \).
\end{proof}
\begin{theorem}\label{th4.3}
	Let \( B \) and \( A \) be two real symmetric matrices. Define
	\[
	W = M^2 (B - A),
	\]
	where \( M \) is a real symmetric and nonsingular matrix. Then, for any vector \( s \in \mathbb{R}^n \) satisfying \( W s \neq 0 \),
	\begin{align}\label{psbqq}
		\frac{\| M (B - A) s \|_2^2}{\| M^{-1} s \|_2^2} \le
		\frac{\| M (B - A) W s \|_2^2}{\| M^{-1} W s \|_2^2}.
	\end{align}
	Equality holds if and only if \( s \) is an eigenvector of \( W \).
\end{theorem}
\begin{proof}
	Let $\bar{s} = M^{-1}s$, $\bar{v} = M^{-1}Ws$. Then we have
	$$\bar{v} = M(B-A)s = M(B-A)M\bar{s}.$$
	Let $K = M(B-A)M$, then 
	$$ \bar{v} = K\bar{s}.$$
	The inequality in (\ref{psbqq}) is equivalent to
	\begin{align}\label{sigma11}
		\frac{\bar{v}^T K^2 \bar{v}}{(\bar{v})^T  \bar{v}} \ge \frac{\bar{s}^T K^2 \bar{s}}{\bar{s}^T \bar{s}}.
	\end{align}
Since \( M \) is symmetric, \( K \) is also symmetric. By Lemma~\ref{4.3}, inequality~\eqref{sigma11} holds, with equality if and only if \( \bar{s} \) is an eigenvector of \( K \). Suppose \( K \bar{s} = \lambda \bar{s} \), then it follows that
\[
M^2 (B - A) s = \lambda s.
\]
This completes the proof.
\end{proof}

The generalized PSB family update \eqref{GPSB2} possesses the least change property and can be derived from the following proposition~\textsuperscript{\cite{dennis1977quasi}}.

\begin{proposition}\label{gpsbles1}
	Let \( B \in \mathbb{R}^{n \times n} \) be a symmetric matrix, and \( M \in \mathbb{R}^{n \times n} \) be a symmetric and nonsingular matrix. Consider the optimization problem
	\[
	\min_{\bar{B}} \| \bar{B} - B \|_{M,F} \quad \text{s.t.} \quad \bar{B} s = y, \quad \bar{B} = \bar{B}^\top,
	\]
	where the weighted Frobenius norm is defined by \(\|X\|_{M,F} = \| M X M \|_F\).
	
	The unique symmetric solution is given by
	\begin{equation}\label{gpsb}
		B_+ = B + \frac{(y - Bs)s^\top M^{-2} + M^{-2}s(y - Bs)^\top}{s^\top M^{-2} s}
		+ \frac{(y - Bs)^\top s}{(s^\top M^{-2} s)^2} M^{-2}s s^\top M^{-2}.
	\end{equation}
\end{proposition}

Assuming \( y = M^{-2} s \), the solution \eqref{gpsb} reduces to the DFP formula:
\begin{equation}\label{DFP}
	B_+ = B + \frac{(y - Bs) y^\top + y (y - Bs)^\top}{s^\top y}
	+ \frac{(y - Bs)^\top s}{(s^\top y)^2} y y^\top.
\end{equation}

Note that throughout, we assume \( A s = y \). Inequality~\eqref{3.1.6} provides a lower bound on the reduction in matrix approximation error for the DFP update:
\[
\frac{\| A^{-\frac{1}{2}} (B - A) s \|_2^2}{\| A^{\frac{1}{2}} s \|_2^2}.
\]

By substituting \( M^{-2} = A \) in Proposition~\ref{th4.3}, we obtain the following improved matrix approximation result corresponding to the operator \( A^{-1} (B - A) \) for the DFP update~\eqref{DFP}.

\begin{corollary}
	Let \( B \in \mathbb{R}^{n \times n} \) be symmetric, and \( A \in \mathbb{R}^{n \times n} \) be symmetric positive definite. Define
	\[
	W = A^{-1} (B - A).
	\]
	Then, for any vector \( s \in \mathbb{R}^n \) satisfying \( W s \neq 0 \),
	\begin{equation}\label{DFPbetter}
	\frac{\| A^{-\frac{1}{2}} (B - A) s \|_2^2}{\| A^{\frac{1}{2}} s \|_2^2} \le
	\frac{\| A^{-\frac{1}{2}} (B - A) W s \|_2^2}{\| A^{\frac{1}{2}} W s \|_2^2}.
\end{equation}
	Equality holds if and only if \( s \) is an eigenvector of \( W \).
\end{corollary}

From the perspective of approximating the inverse of an Hessian matrix \( A^{-1} \), we obtain the following dual result to Proposition~\ref{gpsbles1}.

\begin{proposition}\label{gpsbles2}
	Let \( H, M \in \mathbb{R}^{n \times n} \) be symmetric matrices, with \( M \) nonsingular. Consider the optimization problem
	\begin{align}\label{sp2}
		&\min_{\bar{H}} \| \bar{H} - H \|_{M,F} \nonumber\\
		&\text{s.t.} \quad \bar{H}y = s,\quad \bar{H} = \bar{H}^\top,  \nonumber
	\end{align}
	The unique solution is
	\begin{equation}\label{gpsb2}
		H_+ = H + \frac{(s - Hy)y^\top M^{-2} + M^{-2}y(s - Hy)^\top}{y^\top M^{-2} y}
		+ \frac{(s - Hy)^\top y}{(y^\top M^{-2} y)^2} M^{-2} y y^\top M^{-2}.
	\end{equation}
\end{proposition}

The generalized PSB updates \eqref{gpsb} and \eqref{gpsb2} are dual formulations related by the substitutions \( B \leftrightarrow H \) and \( A \leftrightarrow A^{-1} \). 
Hence, the dual update formula \eqref{gpsb2} for the generalized PSB family, under the condition \( A^{-1} y = s \), admits an error reduction estimate analogous to \eqref{3.1.6}:
\begin{align}\label{bijbds2}
	\| \text{PSB}_{M}(H, A^{-1}, y) - A^{-1} \|_{M, F}^2 
	\le \| H - A^{-1} \|_{M, F}^2 
	- \frac{ \| M(H - A^{-1}) y \|_2^2 }{ \| M^{-1} y \|_2^2 }.
\end{align}
This correspondence also implies that the operator \( M^2(H - A^{-1}) \) in \eqref{gpsb2} arises naturally from replacing \( B - A \) with \( H - A^{-1} \) in the operator \( M^2(B - A) \) associated with \eqref{gpsb}.  It can be interpreted as updating the matrix $H_k$ by applying the following quasi-Newton equation:
\begin{equation}\label{Qk2}
	H_{k+1} W_k y_k = A^{-1} W_k y_k,\nonumber
\end{equation}
where the associated operator is defined by
\begin{equation}\label{Pk2}
	W_k =  M^2(H - A^{-1}).\nonumber
\end{equation}

The improved matrix approximation results for the operator \( M^2(H - A^{-1}) \) associated with the update \eqref{gpsb2} follow directly from the analogous results for the operator \( M^2(B - A) \), and thus no additional proof is required.
\begin{theorem}\label{th4.4}
	Assume \( H \in \mathbb{R}^{n \times n} \) is symmetric, and \( A, M \in \mathbb{R}^{n \times n} \) are symmetric and nonsingular. Define
	\[
	W = M^2 (H - A^{-1}).
	\]
	Then, for any vector \( y \in \mathbb{R}^n \) satisfying \( W y \neq 0 \),
	\begin{align}\label{psbHq}
		\frac{ \| M (H - A^{-1}) y \|_2^2 }{ \| M^{-1} y \|_2^2 }
		\le
		\frac{ \| M (H - A^{-1}) W y \|_2^2 }{ \| M^{-1} W y \|_2^2 }.\nonumber
	\end{align}
	Equality holds if and only if \( y \) is an eigenvector of \( W \).
\end{theorem}

If we assume \( M^{-2}y = s \) in \eqref{gpsb2}, we obtain the BFGS formula:
\begin{equation}\label{1BFGS}
	H_+ = H + \frac{(s - Hy)s^\top + s(s - Hy)^\top}{s^\top y}
	+ \frac{(s - Hy)^\top y}{(s^\top y)^2} ss^\top.
\end{equation}
Note that we assume $A^{-1}y = s$.  Inequality~\eqref{bijbds2} gives an lower bound on the reduction in the matrix approximation error for BFGS:
\[
\frac{\| A^{\frac{1}{2}} (H - A^{-1}) y \|_2^2}{\| A^{-\frac{1}{2}} y \|_2^2}.
\]

Similarly, by replacing \( M^{-2} \) with \( A^{-1} \) in Proposition~\ref{th4.4}, we obtain the improved matrix approximation result corresponding to the operator \( A(H - A^{-1}) \) for the BFGS update~\eqref{1BFGS}.

\begin{corollary}\label{BFGSt}
	Let \( H \in \mathbb{R}^{n \times n} \) be a symmetric matrix, and \( A \in \mathbb{R}^{n \times n} \) be symmetric and positive definite. Define
	\[
	W = A(H - A^{-1}).
	\]
	Then, for any vector \( y \in \mathbb{R}^n \) satisfying \( W y \neq 0 \),
	\begin{align}\label{BFGSbetter}
		\frac{\| A^{\frac{1}{2}} (H - A^{-1}) y \|_2^2}{\| A^{-\frac{1}{2}} y \|_2^2} \le
		\frac{\| A^{\frac{1}{2}} (H - A^{-1}) W y \|_2^2}{\| A^{-\frac{1}{2}} W y \|_2^2}.
	\end{align}
	Equality holds if and only if \( y \) is an eigenvector of \( W \).
\end{corollary}
The operator \( A(H - A^{-1}) \) transforms \( y \) to \[ A(H - A^{-1})y. \] Note that the corresponding new direction of \( s \) becomes
\[
A^{-1} A(H - A^{-1})y = (H - A^{-1})y = (H - A^{-1})As = H(A - H^{-1})s.
\]
This shows that applying the operator \( A(H - A^{-1}) \) to the BFGS update~\eqref{1BFGS} is equivalent to applying the operator \( B^{-1}(B - A) \) to the BFGS update
\begin{equation}\label{bfsgB}
	B_+ = B - \frac{Bs s^\top B}{s^\top B s} + \frac{y y^\top}{s^\top y}.
\end{equation}
Since the quasi-Newton equation remains unchanged upon replacing \( s \) by \( -s \), Corollary~\ref{BFGSt} can also be interpreted as the improved matrix approximation result corresponding to the operator \( B^{-1}(A - B) \) for the BFGS update~\eqref{bfsgB}.

We now present the improved matrix approximation result corresponding to the operator \( B^{-1}(B - A) \) for the DFP update~\eqref{DFP}. Before that, we prove the following lemma.

\begin{lemma}\label{3.10}
	Let \( L \in \mathbb{R}^{n \times n} \) be symmetric and positive definite, and let \( u \in \mathbb{R}^n \) satisfy \( Lu \neq u \). If \( L \succeq I \) or \( L \preceq I \), then
	\begin{equation}\label{lu}
		\frac{\| (L^{-1} - I)(I - L)u \|_2^2}{\| (I - L)u \|_2^2} \ge \frac{\| (L^{-1} - I)u \|_2^2}{\| u \|_2^2},
	\end{equation}
	and equality holds if and only if \( u \) is an eigenvector of \( L \).
\end{lemma}

\begin{proof}
	Since \( L \in \mathbb{R}^{n \times n} \) is symmetric, it admits an orthogonal eigendecomposition. Let
	\[
	u = \sum_{i=1}^{n} u^i,
	\]
	where each \( u^i \) denotes the component of \( u \) along the \( i \)-th eigenvector, satisfying
	\[
	L u^i = \lambda_i u^i, \quad \langle u^i, u^j \rangle = 0 \quad \text{for } i \ne j.
	\]
	
	Then proving~\eqref{lu} is equivalent to verifying the inequality
	\[
	\frac{\sum_{i=1}^{n} \left( \frac{1}{\lambda_i} - 1 \right)^2 (1 - \lambda_i)^2 \| u^i \|^2}{\sum_{i=1}^{n} (1 - \lambda_i)^2 \| u^i \|^2} \ge \frac{\sum_{i=1}^{n} \left( \frac{1}{\lambda_i} - 1 \right)^2 \| u^i \|^2}{\sum_{i=1}^{n} \| u^i \|^2}.
	\]
	
	This is equivalent to
	\begin{equation}\label{pai2}
		\sum_{i=1}^{n} \left( \frac{1}{\lambda_i} - 1 \right)^2 (1 - \lambda_i)^2 \| u^i \|^2 \sum_{i=1}^{n} \| u^i \|^2 \ge \sum_{i=1}^{n} (1 - \lambda_i)^2 \| u^i \|^2 \sum_{i=1}^{n} \left( \frac{1}{\lambda_i} - 1 \right)^2 \| u^i \|^2.
	\end{equation}
	
	Expanding both sides of~\eqref{pai2} and simplifying yields
	\begin{align}
		&\sum_{1 \le i < j \le n} \left[ \left( \frac{1}{\lambda_i} - 1 \right)^2 (1 - \lambda_i)^2 + \left( \frac{1}{\lambda_j} - 1 \right)^2 (1 - \lambda_j)^2 \right] \| u^i \|^2 \| u^j \|^2 \nonumber \\
		&\quad \ge \sum_{1 \le i < j \le n} \left[ \left( \frac{1}{\lambda_i} - 1 \right)^2 (1 - \lambda_j)^2 + \left( \frac{1}{\lambda_j} - 1 \right)^2 (1 - \lambda_i)^2 \right] \| u^i \|^2 \| u^j \|^2. \label{mon1}
	\end{align}

	Since \( L \succeq I \) or \( L \preceq I \), all eigenvalues \( \lambda_i \) lie either in the interval \([1, +\infty)\) or in \((0, 1]\). On each of these intervals, the functions 
	\[
	f_1(\lambda)=\left(\frac{1}{\lambda} - 1\right)^2 \quad \text{and} \quad f_2(\lambda)=(1 - \lambda)^2
	\]
	have the same monotonicity behavior. Hence, for any pair \( \lambda_i, \lambda_j \), the values \( f_1(\lambda_i) \) and \( f_1(\lambda_j) \) are ordered in the same way as \( f_2(\lambda_i) \) and \( f_2(\lambda_j) \). Hence, by the rearrangement inequality, we have
	\begin{align*}
		&\left( \frac{1}{\lambda_i} - 1 \right)^2 (1 - \lambda_i)^2 + \left( \frac{1}{\lambda_j} - 1 \right)^2 (1 - \lambda_j)^2 \\
		&\quad \ge \left( \frac{1}{\lambda_i} - 1 \right)^2 (1 - \lambda_j)^2 + \left( \frac{1}{\lambda_j} - 1 \right)^2 (1 - \lambda_i)^2.
	\end{align*}
	Therefore, inequality~\eqref{mon1} holds, with equality if and only if \( \lambda_i = \lambda_j \) for all \( i \ne j \). This completes the proof.
\end{proof}
\begin{theorem}\label{th3.11}
	Let \( A, B \in \mathbb{R}^{n \times n} \) be symmetric and positive definite. Define
	\[
	W = B^{-1} (B - A).
	\]
	Assume that \( B \succeq A \) or \( B \preceq A \). Then, for any vector \( s \in \mathbb{R}^n \) satisfying \( Ws \neq 0 \), inequality~\eqref{DFPbetter} holds, with equality if and only if \( s \) is an eigenvector of \( W \).
\end{theorem}

\begin{proof}
	Let \( \bar{s} = A^{\frac{1}{2}}s \), \( L = A^{\frac{1}{2}} B^{-1} A^{\frac{1}{2}} \), and \( K = A^{-\frac{1}{2}}(B - A)A^{-\frac{1}{2}} \). Then \eqref{DFPbetter} is equivalent to
	\begin{equation}\label{p311.1}
		\frac{\| KLK\bar{s} \|_2^2}{\| LK\bar{s} \|_2^2} \ge \frac{\| K\bar{s} \|_2^2}{\| \bar{s} \|_2^2}.
	\end{equation}
	Note that \( K = L^{-1} - I \). Therefore, \eqref{p311.1} is equivalent to
	\begin{equation}\label{p311.2}
		\frac{\| (L^{-1} - I)(I - L)\bar{s} \|_2^2}{\| (I - L)\bar{s} \|_2^2} \ge \frac{\| (L^{-1} - I)\bar{s} \|_2^2}{\| \bar{s} \|_2^2}.
	\end{equation}
	Since \( B \succeq A \) or \( B \preceq A \), it follows that \( L \preceq  I \) or \( L \succeq I \). Moreover, the assumption \( Ws \ne 0 \) implies \( B^{-1}A s \ne s \), and thus \( L \bar{s} \ne \bar{s} \). By Lemma~\ref{3.10}, inequality~\eqref{p311.2} follows, with equality if and only if \( \bar{s} \) is an eigenvector of \( L \). 
	
	If equality holds, there exists \( \lambda \in \mathbb{R} \) such that \( L\bar{s} = \lambda \bar{s} \), which implies \( B^{-1}A s = \lambda s \), and hence \( s \) is an eigenvector of \( W \). This completes the proof.
\end{proof}

 The operator \( A^{-1}(B - A) \) transforms \( s \) to
 \[
 A^{-1}(B - A)s.
 \]
 Then, the corresponding new direction of \( y \) becomes
 \[
 AA^{-1} (B - A)s = (B - A)s = (B - A)A^{-1}y = B(A^{-1} - B^{-1})y.
 \]
 This shows that applying the operator \( A^{-1}(B - A) \) to the BFGS update~\eqref{bfsgB} is equivalent to applying the operator \( H^{-1}(H - A^{-1}) \) or \( H^{-1}(A^{-1} - H) \) to the BFGS update~\eqref{1BFGS}.  Hence, the improved matrix approximation result corresponding to the operator \( A^{-1}(B - A) \) for the BFGS update~\eqref{bfsgB} can be derived analogously. By Theorem~\ref{th3.11} and duality, we state the result without proof.
 
 \begin{theorem}\label{th3.12}
 	Let \( A, H \in \mathbb{R}^{n \times n} \) be symmetric and positive definite. Define
 	\[
 	W = H^{-1}(H - A^{-1}).
 	\]
 	Assume that \( H \succeq A^{-1} \) or \( H \preceq A^{-1} \). Then, for any vector \( y \in \mathbb{R}^n \) satisfying \( W y \neq 0 \), inequality~\eqref{BFGSbetter} holds, with equality if and only if \( y \) is an eigenvector of \( W \).
 \end{theorem}
The BGM update formula \eqref{B1} satisfies the following matrix approximation error reduction property.

\begin{proposition}
	Let \( B, A \in \mathbb{R}^{n \times n} \), and let \( s \in \mathbb{R}^n \) be a nonzero vector. For the BGM update formula \eqref{B1} with \( A s = y \), the following approximation property holds:
	\begin{equation}\label{BGMd}
		\| \mathrm{BGM}(B, A, s) - A \|_F^2 = \| B - A \|_F^2 - \frac{\| (B - A) s \|_2^2}{\| s \|_2^2}.
	\end{equation}
\end{proposition}

\begin{proof}
	From the BGM update formula \eqref{B1} and the relation \( A s = y \), we have
	\begin{align*}
		\| \mathrm{BGM}(B, A, s) - A \|_F^2 
		&= \left\| (B - A) - \frac{(B - A) s s^\top}{\| s \|^2} \right\|_F^2 \\
		&= \| (B - A) \left( I - \frac{s s^\top}{\| s \|^2} \right) \|_F^2 \\
		&= \mathrm{trace} \left[ \left( I - \frac{s s^\top}{\| s \|^2} \right) (B - A)^\top (B - A) \left( I - \frac{s s^\top}{\| s \|^2} \right) \right] \\
		&= \mathrm{trace} \left[ (B - A)^\top (B - A) \left( I - \frac{s s^\top}{\| s \|^2} \right) \right] \\
		&= \mathrm{trace} \big[ (B - A)^\top (B - A) \big] - \frac{\mathrm{trace} \big[ (B - A)^\top (B - A) s s^\top \big]}{\| s \|^2} \\
		&= \| B - A \|_F^2 - \frac{s^\top (B - A)^\top (B - A) s}{\| s \|^2} \\
		&= \| B - A \|_F^2 - \frac{\| (B - A) s \|_2^2}{\| s \|_2^2}.
	\end{align*}
	This establishes \eqref{BGMd} and completes the proof.
\end{proof}

Finally, we present the improved matrix approximation result corresponding to the operator \( W = (B - A)^\top \) for the BGM update.

\begin{theorem}
	Let \( A, B \in \mathbb{R}^{n \times n} \) and define
	\[
	W = (B - A)^\top.
	\]
	Then, for any vector \( s \in \mathbb{R}^n \) such that \( W s \neq 0 \), the following inequality holds:
	\begin{equation}\label{BGMdd}
		\frac{\| (B - A) W s \|_2^2}{\| W s \|_2^2} \ge \frac{\| (B - A) s \|_2^2}{\| s \|_2^2},
	\end{equation}
	with equality if and only if \( s \) is an eigenvector of \( W^\top W \).
\end{theorem}

\begin{proof}
	Expanding the square in inequality \eqref{BGMdd} and rearranging terms yields the equivalent inequality
	\begin{equation}\label{cau}
		(W^\top W  s)^\top ( W^\top W s) \cdot s^\top s \ge (s^\top W^\top W  s)^2.
	\end{equation}
	By the Cauchy-Schwarz inequality, \eqref{cau} holds, with equality if and only if \( W^\top W s \) and \( s \) are collinear, i.e., there exists \( \lambda \in \mathbb{R} \) such that
	\[
W^\top W s = \lambda s.
	\]
	This completes the proof.
\end{proof}

\subsection{Proposition on Matrix Approximation for the Projection Operator}
This subsection presents improved matrix approximation properties for projection operators, starting again with the generalized PSB family.

\begin{theorem}\label{th3.2.1}
	Let \( M \in \mathbb{R}^{n \times n} \) be symmetric positive definite, and let \( A, B \in \mathbb{R}^{n \times n} \). Suppose \( \tilde{s} \) is the projection of a vector \( s \in \mathbb{R}^n \) onto  \( \ker(B - A)^{\perp_{M^{-2}}} \), and assume \( \tilde{s} \neq 0 \). Then, the following improved matrix approximation result holds for the generalized PSB family~\eqref{gpsb}:
	\begin{equation}\label{psbq2}
		\frac{\| M (B - A) s \|_2^2}{\| M^{-1} s \|_2^2} \le
		\frac{\| M (B - A) \tilde{s} \|_2^2}{\| M^{-1}  \tilde{s} \|_2^2},
	\end{equation}
	with equality if and only if \( s \in \ker(B - A)^{\perp_{M^{-2}}} \).
\end{theorem}

\begin{proof}
	Since \( \tilde{s} \) is the projection of \( s \in \mathbb{R}^n \) onto  \( \ker(B - A)^{\perp_{M^{-2}}} \), we have
	\[
	\tilde{s} - s \in \left( \ker(B - A)^{\perp_{M^{-2}}} \right)^{\perp_{M^{-2}}} = \ker(B - A).
	\]
	Thus,
	\begin{equation}\label{eq3.2.1}
		\| M (B - A) \tilde{s} \|_2^2 = \| M (B - A) s + M (B - A)(\tilde{s} - s) \|_2^2 = \| M (B - A) s \|_2^2.
	\end{equation}
By the property of orthogonal projections, it follows that
\begin{equation}\label{eq3.2.2}
	\| \tilde{s} \|_{M^{-2}}^2 + \| \tilde{s} - s \|_{M^{-2}}^2 = \| s \|_{M^{-2}}^2,\nonumber
\end{equation}
where
\[
\| x \|_{M^{-2}} := \| M^{-1} x \|_2.
\]
Therefore,
\begin{equation}\label{eq3.2.3}
	\| M^{-1} \tilde{s} \|_2^2 = \| \tilde{s} \|_{M^{-2}}^2 \le \| s \|_{M^{-2}}^2 = \| M^{-1} s \|_2^2.
\end{equation}
Equality in~\eqref{eq3.2.3} holds if and only if \( s = \tilde{s} \), which is equivalent to \( s \in \ker(B - A)^{\perp_{M^{-2}}} \). Combining~\eqref{eq3.2.1} and~\eqref{eq3.2.3} yields the result~\eqref{psbq2}. This completes the proof.
\end{proof}

\begin{remark}\label{re3.16}
	Under the inner product \( \langle \cdot, \cdot \rangle_{M^{-2}} \), the space \( \ker(B - A) \) is orthogonal to \( \ker(B - A)^{\perp_{M^{-2}}} \). From the proof of Theorem~\ref{th3.2.1} and the property of orthogonal projections, we obtain the geometric interpretation:
	\begin{equation}\label{psbq3}
		\frac{\| M (B - A) s \|_2^2}{\| M^{-1} s \|_2^2}
		= \frac{\| \tilde{s} \|_{M^{-2}}^2}{\| s \|_{M^{-2}}^2}
		\cdot \frac{\| M (B - A) \tilde{s} \|_2^2}{\| M^{-1} \tilde{s} \|_2^2}
		= \sin^2 \theta \cdot \frac{\| M (B - A) \tilde{s} \|_2^2}{\| M^{-1} \tilde{s} \|_2^2},\nonumber
	\end{equation}
where \( \theta \) denotes the angle between \( s \) and  \( \ker(B - A) \) with respect to the inner product \( \langle \cdot, \cdot \rangle_{M^{-2}} \).
\end{remark}

Based on Theorem~\ref{th3.2.1} and the correspondence between  DFP  and the generalized PSB family~\eqref{GPSB2}, we immediately obtain the improved matrix approximation result for the DFP method.

\begin{corollary}\label{th3.2.2}
	Let \( A \in \mathbb{R}^{n \times n} \) be symmetric positive definite, and \( B \in \mathbb{R}^{n \times n} \). Suppose \( \tilde{s} \) is the projection of \( s \in \mathbb{R}^n \) onto  \( \ker(B - A)^{\perp_A} \), and assume \( \tilde{s} \neq 0 \). Then the following improved matrix approximation result holds for the DFP method:
	\begin{equation}\label{dfp2}
		\frac{\| A^{-\frac{1}{2}} (B - A) s \|_2^2}{\| A^{\frac{1}{2}} s \|_2^2} \le
		\frac{\| A^{-\frac{1}{2}} (B - A) \tilde{s} \|_2^2}{\| A^{\frac{1}{2}} \tilde{s} \|_2^2},\nonumber
	\end{equation}
	with equality if and only if \( s \in \ker(B - A)^{\perp_A} \).
\end{corollary}

By duality, we state without proof the improved matrix approximation result for the generalized PSB update~\eqref{gpsb2}.

\begin{theorem}\label{th3.2.4}
	Let \( A, M \in \mathbb{R}^{n \times n} \) be symmetric positive definite, and let \( H \in \mathbb{R}^{n \times n} \). Suppose \( \tilde{y} \) is the projection of \( y \in \mathbb{R}^n \) onto  \( \ker(H - A^{-1})^{\perp_{M^{-2}}} \), and assume \( \tilde{y} \neq 0 \). Then the following improved matrix approximation result holds for the generalized PSB update~\eqref{gpsb2}:
	\begin{equation}\label{psbHq2}
		\frac{\| M (H - A^{-1}) y \|_2^2}{\| M^{-1} y \|_2^2} \le
		\frac{\| M (H - A^{-1}) \tilde{y} \|_2^2}{\| M^{-1} \tilde{y} \|_2^2},\nonumber
	\end{equation}
	with equality if and only if \( y \in \ker(H - A^{-1})^{\perp_{M^{-2}}} \).
\end{theorem}

Based on Theorem~\ref{th3.2.4} and the correspondence between  BFGS  and the generalized PSB family~\eqref{gpsb2}, we immediately obtain the improved matrix approximation result for the BFGS method.

\begin{corollary}\label{th3.2.5}
	Let \( A \in \mathbb{R}^{n \times n} \) be symmetric positive definite, and \( H \in \mathbb{R}^{n \times n} \). Suppose \( \tilde{y} \) is the projection of \( y \in \mathbb{R}^n \) onto  \( \ker(H - A^{-1})^{\perp_{A^{-1}}} \), and assume \( \tilde{y} \neq 0 \). Then the following improved matrix approximation result holds for the BFGS update~\eqref{1BFGS}:
	\begin{equation}\label{bfgs2}
		\frac{\| A^{-\frac{1}{2}} (H - A^{-1}) y \|_2^2}{\| A^{\frac{1}{2}} y \|_2^2} \le
		\frac{\| A^{-\frac{1}{2}} (H - A^{-1}) \tilde{y} \|_2^2}{\| A^{\frac{1}{2}} \tilde{y} \|_2^2},\nonumber
	\end{equation}
	with equality if and only if \( y \in \ker(H - A^{-1})^{\perp_{A^{-1}}} \).
\end{corollary}

The following proposition shows that Corollary~\ref{th3.2.5} can also be interpreted as an improved matrix approximation result of the BFGS update~\eqref{bfsgB} when \( s \) is projected onto  \( \ker(B - A)^{\perp_A} \).

\begin{proposition}
	Let \( A \in \mathbb{R}^{n \times n} \) be symmetric positive definite, and \( B \in \mathbb{R}^{n \times n} \) be nonsingular. If \( \tilde{s} \) is the projection of \( s \in \mathbb{R}^n \) onto  \( \ker(B - A)^{\perp_A} \), then \( A\tilde{s} \) is the projection of \( As \) onto  \( \ker(B^{-1} - A^{-1})^{\perp_{A^{-1}}} \).
\end{proposition}

\begin{proof}
	By the definition of projection, the proposition is equivalent to proving that for all \( u \in \ker(B^{-1} - A^{-1})^{\perp_{A^{-1}}} \),
	\begin{equation}\label{eq3.2.5}
		\| As - u \|_{A^{-1}} \ge \| As - A\tilde{s} \|_{A^{-1}}.
	\end{equation}
	Since \( \tilde{s} \) is the projection of \( s \in \mathbb{R}^n \) onto  \( \ker(B - A)^{\perp_A} \), we have for all \( v \in \ker(B - A)^{\perp_A} \),
	\begin{equation}\label{eq3.2.4}
		\| s - v \|_A \ge \| s - \tilde{s} \|_A.
	\end{equation}
	Inequality~\eqref{eq3.2.5} is equivalent to showing that for all \( u \in \ker(B^{-1} - A^{-1})^{\perp_{A^{-1}}} \),
	\begin{equation}\label{eq3.2.6}
		\| s - A^{-1} u \|_A \ge \| s - \tilde{s} \|_A.\nonumber
	\end{equation}
	Therefore, by~\eqref{eq3.2.4}, it suffices to show that \( A^{-1} u \in \ker(B - A)^{\perp_A} \).
	
	From Proposition~\ref{p2}, we have
	\begin{equation}\label{eq3.2.7}
		u \in \ker(B^{-1} - A^{-1})^{\perp_{A^{-1}}} = \operatorname{Im} A(B^{-1} - A^{-1})^T.
	\end{equation}
	By~\eqref{eq3.2.7}, Proposition~\ref{p2}, and the nonsingularity of \( B \), we obtain
	\begin{align*}
		A^{-1} u &\in \operatorname{Im}(B^{-1} - A^{-1})^T\\ &= \operatorname{Im}(B^{-1} - A^{-1})^T B^T \\
		&=\operatorname{Im} (I-A^{-1}B^T)\\
		&= \operatorname{Im} A^{-1}(A - B^T) \\
		&= \operatorname{Im} A^{-1}(B - A)^T \\
		&= \ker(B - A)^{\perp_A}.
	\end{align*}
	This completes the proof.
\end{proof}

The following theorem provides an improved matrix approximation result for the BGM update when a projection operator is applied.

\begin{theorem}\label{th3.2.6}
	Let \( A, B \in \mathbb{R}^{n \times n} \). Suppose that \( \tilde{s} \) is the projection of \( s \in \mathbb{R}^n \) onto  \( \ker(B - A)^{\perp} \), and that \( \tilde{s} \neq 0 \). Then, the following improved matrix approximation result for the BGM update holds:
	\begin{equation}\label{bgm2}
		\frac{\| (B - A)s \|_2^2}{\| s \|_2^2} \le
		\frac{\| (B - A)\tilde{s} \|_2^2}{\| \tilde{s} \|_2^2},
	\end{equation}
	with equality if and only if \( s \in \ker(B - A)^{\perp} \).
\end{theorem}

\begin{proof}
	Since \( \tilde{s} \) is the projection of \( s \in \mathbb{R}^n \) onto  \( \ker(B - A)^{\perp} \), it follows that
	\[
	\tilde{s} - s \in \left(\ker(B - A)^{\perp}\right)^{\perp} = \ker(B - A).
	\]
	Therefore,
	\begin{equation}\label{eq3.3.1}
		\| (B - A)\tilde{s} \|_2^2 = \| (B - A)s + (B - A)(\tilde{s} - s) \|_2^2 = \| (B - A)s \|_2^2.
	\end{equation}
	
	By the property of orthogonal projections, we have
	\begin{equation}\label{eq3.3.2}
		\| \tilde{s} \|_2^2 + \| \tilde{s} - s \|_2^2 = \| s \|_2^2.\nonumber
	\end{equation}
	Hence,
	\begin{equation}\label{eq3.3.3}
		\| \tilde{s} \|_2^2 \le \| s \|_2^2,
	\end{equation}
	with equality if and only if \( s = \tilde{s} \), which is equivalent to \( s \in \ker(B - A)^{\perp} \).
	
	Combining \eqref{eq3.3.1} and \eqref{eq3.3.3}, inequality~\eqref{bgm2} holds. This completes the proof.
\end{proof}

From \eqref{eq2.7.2} in Proposition~\ref{p4}, the orthogonalization method proposed in Subsection~\ref{Projection operator approaches} can be interpreted as projecting \( s \) onto the orthogonal complement of a certain subspace of \(\ker(B - A)\).
 The following result shows that this approach still yields an improved matrix approximation result.

\begin{theorem}\label{th3.2.7}
	Let \( M \in \mathbb{R}^{n \times n} \) be symmetric and positive definite, and let \( A, B \in \mathbb{R}^{n \times n} \). Let \( S \) be a subspace of \( \ker(B - A) \). Suppose that \( \tilde{s} \) is the projection of \( s \in \mathbb{R}^n \) onto  \( S^{\perp_{M^{-2}}} \), and that \( \tilde{s} \neq 0 \). Then, the following improved matrix approximation result for the generalized PSB family~\eqref{GPSB2} holds:
	\begin{equation}\label{psbq4}
		\frac{\| M(B - A)s \|_2^2}{\| M^{-1}s \|_2^2} \le
		\frac{\| M(B - A)\tilde{s} \|_2^2}{\| M^{-1}\tilde{s} \|_2^2},
	\end{equation}
	with equality if and only if \( s \in S^{\perp_{M^{-2}}} \).
\end{theorem}

\begin{proof}
	Since \( \tilde{s} \) is the projection of \( s \in \mathbb{R}^n \) onto  \( S^{\perp_{M^{-2}}} \), we have
	\[
	\tilde{s} - s \in \big(S^{\perp_{M^{-2}}}\big)^{\perp_{M^{-2}}} = S \subseteq \ker(B - A).
	\]
	Therefore,
	\begin{equation}\label{eq3.4.1}
		\| M(B - A)\tilde{s} \|_2^2
		= \| M(B - A)s + M(B - A)(\tilde{s} - s) \|_2^2
		= \| M(B - A)s \|_2^2.
	\end{equation}
	
By the property of orthogonal projections, we have
\begin{equation}\label{eq3.4.2}
	\| \tilde{s} \|_{M^{-2}}^2 + \| \tilde{s} - s \|_{M^{-2}}^2 = \| s \|_{M^{-2}}^2.
\end{equation}
Hence,
\begin{equation}\label{eq3.4.3}
	\| M^{-1} \tilde{s} \|_2^2 = \| \tilde{s} \|_{M^{-2}}^2 \le \| s \|_{M^{-2}}^2 = \| M^{-1} s \|_2^2.
\end{equation}

	Equality in \eqref{eq3.4.2} holds if and only if \( s = \tilde{s} \), which is equivalent to \( s \in S^{\perp_{M^{-2}}} \). Combining \eqref{eq3.4.1} and \eqref{eq3.4.3}, inequality~\eqref{psbq4} follows. This completes the proof.
\end{proof}

\begin{remark}\label{re3.23}
	Similarly to the conclusion in Remark~\ref{re3.16}, we have
	\begin{equation}
		\frac{\| M (B - A) s \|_2^2}{\| M^{-1} s \|_2^2}
		= \frac{\| \tilde{s} \|_{M^{-2}}^2}{\| s \|_{M^{-2}}^2}
		\cdot \frac{\| M (B - A) \tilde{s} \|_2^2}{\| M^{-1} \tilde{s} \|_2^2}
		= \sin^2 \gamma \cdot \frac{\| M (B - A) \tilde{s} \|_2^2}{\| M^{-1} \tilde{s} \|_2^2},\nonumber
	\end{equation}
	where \(\gamma\) denotes the angle between the vector \( s \) and the subspace \( S \).
	
	Let \(\theta\) be the angle between \( s \) and  \(\ker(B - A)\).
	Since \( S \) is a subspace of \(\ker(B - A)\), it is straightforward to show that
	\begin{equation}\label{thera}
		\sin \theta \le \sin \gamma.
	\end{equation}
	Recall the stepwise orthogonalization method \eqref{or2}, which projects \( s_k \) onto the subspace 
	\[\operatorname{span}\{ \tilde{s}_{k-d}, \ldots, \tilde{s}_{k-1} \}. \] 
	Inequality \eqref{thera} indicates that reducing the subspace dimension \( d \) generally weakens the matrix approximation quality.
\end{remark}

Since the improved matrix approximation results for DFP and BFGS methods using projection operators are derived from the generalized PSB family, projecting \( s \) onto the orthogonal complement of a subspace of \( \ker(B - A) \) still preserves the improved matrix approximation property. For the BFGS method in particular, we only need to supplement the following proposition.

\begin{proposition}
	Let \( A \in \mathbb{R}^{n \times n} \) be symmetric positive definite, and \( B \in \mathbb{R}^{n \times n} \) be nonsingular. Suppose \( S \) is a subspace of \( \ker(B - A) \), and \( \tilde{s} \) is the projection of \( s \in \mathbb{R}^n \) onto \( S^{\perp_{A}} \). Then \( AS \) is a subspace of \( \ker(B^{-1} - A^{-1}) \), and \( A\tilde{s} \) is the projection of \( As \) onto \( (AS)^{\perp_{A^{-1}}} \).
\end{proposition}

\begin{proof}
	For any \( u \in AS \), there exists \( v \in S \subseteq \ker(B - A) \) such that \( u = Av \). Then,
	\[
	(B^{-1} - A^{-1})u = (B^{-1} - A^{-1})Av = B^{-1}(A - B)v = 0,
	\]
	which shows that \( AS \subseteq \ker(B^{-1} - A^{-1}) \).
	
	To complete the proof, it suffices to show that for all \( u \in (AS)^{\perp_{A^{-1}}} \),
	\begin{equation}\label{eq3.4.5}
		\| As - u \|_{A^{-1}} \ge \| As - A\tilde{s} \|_{A^{-1}}.
	\end{equation}
	Inequality~\eqref{eq3.4.5} is equivalent to
	\begin{equation}\label{eq3.4.6}
		\| s - A^{-1}u \|_A \ge \| s - \tilde{s} \|_A.\nonumber
	\end{equation}
	Since \( \tilde{s} \) is the projection of \( s \) onto \( S^{\perp_{A}} \), it suffices to prove that 
	\begin{equation}\label{eq3.23}
		A^{-1}u \in S^{\perp_A}, \quad \forall u\in(AS)^{\perp_{A^{-1}}}.
	\end{equation}
	For each $v\in S$, we have
	\[
	\langle u, Av \rangle_{A^{-1}} = 0 \quad \Rightarrow \quad \langle u, v \rangle = 0 \quad \Rightarrow \quad \langle A^{-1}u, v \rangle_A = 0,
	\]
	which implies \eqref{eq3.23}. This completes the proof.
\end{proof}

Finally, we present the improved matrix approximation result for the BGM method when \( s \) is projected onto the orthogonal complement of a subspace of \( \ker(B - A) \).

\begin{theorem}\label{th3.2.8}
	Let \( A, B \in \mathbb{R}^{n \times n} \), and let \( S \) be a subspace of \( \ker(B - A) \). Suppose \( \tilde{s} \) is the projection of \( s \in \mathbb{R}^n \) onto \( S^{\perp} \), with \( \tilde{s} \neq 0 \). Then the following improved matrix approximation holds for BGM:
	\begin{equation}\label{bgm3}
		\frac{\| (B - A) s \|_2^2}{\| s \|_2^2} \le
		\frac{\| (B - A) \tilde{s} \|_2^2}{\| \tilde{s} \|_2^2},
	\end{equation}
	with equality if and only if \( s \in S^{\perp} \).
\end{theorem}
\begin{proof}
	Since \( \tilde{s} \) is the projection of \( s \) onto \( S^{\perp} \), it follows that
	\[
	\tilde{s} - s \in (S^{\perp})^{\perp} = S \subseteq \ker(B - A).
	\]
	Therefore,
	\begin{equation}\label{eq3.3.10}
		\| (B - A) \tilde{s} \|_2^2 = \| (B - A) s + (B - A)(\tilde{s} - s) \|_2^2 = \| (B - A) s \|_2^2.
	\end{equation}
	
	By the property of orthogonal projections,
	\begin{equation}\label{eq3.3.30}
		\| \tilde{s} \|_2 \le \| s \|_2,
	\end{equation}
	with equality if and only if \( s = \tilde{s} \), i.e., \( s \in S^{\perp} \).
	
	Combining \eqref{eq3.3.10} and \eqref{eq3.3.30} yields inequality~\eqref{bgm3}. This completes the proof.
\end{proof}

\section{Algorithmic Implementation of the Operator Approaches}\label{section4}
In Subsection~\ref{Image operator approaches}, we constructed the corresponding image operators for the Broyden family, the generalized PSB family, and the BGM method, respectively. They are given as follows:
\begin{itemize}
	\item \textbf{Broyden family:} \( W_k = B_k^{-1}(B_k - A) \) or \( W_k = A^{-1}(B_k - A) \).
	\item \textbf{Generalized PSB family:} \( W = M^2(B_k - A) \), where \( M \) is the symmetric nonsingular matrix defining the update.
	\item \textbf{BGM:} \( W = (B_k-A)^T \) (the identity matrix).
\end{itemize}

Unlike the standard quasi-Newton equation~\eqref{standardQN2}, 
the image operator approach replaces it with the modified quasi-Newton equation~\eqref{Qk}. For the Broyden family, the practical image operator is \(W= B_k^{-1}(B_k - A) \), whose mapping on \( s_k \) can be computed by
\[
B_k^{-1}(B_k - A)s_k = s_k - B_k^{-1}y_k.
\]

For the generalized PSB family, the mapping of \( W = M^2(B_k - A) \) on \( s_k \) can be  computed by
\[
M^2(B_k - A) s_k = M^2(B_ks_k-y_k)=M^2[(1-\alpha_k) g_k - g_{k+1}].
\]

Therefore, for general nonlinear functions, we propose image operator quasi-Newton method for the Broyden family in Algorithm~\ref{al-3-4}.

\begin{algorithm}[!htbp]
	\caption{Image Operator Quasi-Newton Method for the Broyden Family}\label{al-3-4}
	\begin{algorithmic}[1]
		\State Set \( k := 0 \), choose an initial point \( x_0 \in \mathbb{R}^n \), and initialize a symmetric positive definite matrix \( B_0 \in \mathbb{R}^{n \times n} \). 
		 Specify a tolerance \( \epsilon \ge 0 \). Compute initial gradient \( g_0 := \nabla f(x_0) \).
		\State If \( \|g_k\| \le \epsilon \), terminate the algorithm.
		\State Update the iterate:
		\[
		x_{k+1} := x_k - \alpha_k B_k^{-1} g_k,
		\]
		where \( \alpha_k \) is determined by a line search.
		\State Compute \( s_k := x_{k+1} - x_k \), \( y_k := g_{k+1} - g_k \).
		\State Compute \( u_k := s_k - B_k^{-1}y_k \); choose an appropriate \( t_k > 0 \), and evaluate
	   \begin{equation}\label{v_kjis}
				v_k := \frac{g(x_{k+1} + t_k u_k) - g(x_{k+1})}{t_k}.
		\end{equation}

		\State Update \( B_{k+1} \) by Broyden family update (\ref{Broyden2}) ensuring the quasi-Newton equation:
		\begin{equation}\label{secant1-4}
			B_{k+1} u_k = v_k.\nonumber
		\end{equation}
		Set \( k := k + 1 \) and return to Step 2.
	\end{algorithmic}
\end{algorithm}
\begin{remark}
	Algorithm~\ref{al-3-4} can be straightforwardly extended to the L-BFGS method. By utilizing the two-loop recursion of L-BFGS to compute \( H_k y_k \), the vector \( u_k \) is obtained. Then, \( v_k \) is computed as in \eqref{v_kjis}, and the pair \((u_k, v_k)\) is stored accordingly.
\end{remark}

For general nonlinear functions, we propose image operator quasi-Newton method for the generlized PSB family in Algorithm~\ref{al-4-4}.
\begin{algorithm}[!htbp]
	\caption{Image Operator Quasi-Newton Method for the Generalized PSB Family}\label{al-4-4}
	\begin{algorithmic}[1]
		\State Set \( k := 0 \), choose an initial point \( x_0 \in \mathbb{R}^n \), and initialize a symmetric positive definite matrix \( B_0 \in \mathbb{R}^{n \times n} \). 
		Specify a tolerance \( \epsilon \ge 0 \). 
		\State If $\|g_k\| \le \epsilon$, stop the algorithm.
		\State Update
		\begin{align}
			x_{k+1} := x_k - \alpha_k B_k^{-1} g_k,\nonumber
		\end{align}
		where $\alpha_k$ is determined by line search.
		\State Choose a symmetric non-singular $M_k$, and compute
		$$u_k := M_k^2[(1 - \alpha_k) g_k - g_{k+1}];$$ choose an appropriate $t_k > 0$, and compute
	     \[v_k := \frac{g(x_{k+1} + t_k u_k) - g(x_{k+1})}{t_k}.\]
        \State According to the quasi-Newton equation
		\begin{equation}\label{secant1-5}
			B_{k+1} u_k = v_k\nonumber
		\end{equation}
		and the generalized PSB update (\ref{GPSB2}), update $B_{k+1}$. Let $k := k + 1$. Go to step 2.
	\end{algorithmic}
\end{algorithm}
\begin{remark}
	If the curvature condition \( u_k^\top v_k > 0 \) is not satisfied in Algorithm~\ref{al-3-4} or Algorithm~\ref{al-4-4}, fall back to the standard quasi-Newton equation by setting \( u_k = s_k \) and \( v_k = y_k \).
	The parameter \( t_k \) is typically chosen to be sufficiently small, or such that \( \|t_k u_k\| \) is approximately comparable to \( \|s_k\| \).
\end{remark}

Recall from Subsection~\ref{Projection operator approaches} that, instead of computing the projection of \( s_k \) onto \( \ker(B_k - A)^{\perp_W} \), we obtain its projection onto \( \operatorname{span}\{\tilde{s}_{k-d}, \ldots, \tilde{s}_{k-1}\}^{\perp_W} \) via the stepwise orthogonalization process~\eqref{or2}. The corresponding orthogonalization coefficients for the Broyden family, the generalized PSB family, and BGM are given in~\eqref{O1}, \eqref{O2}, and~\eqref{O3}, respectively.  
Owing to the similarity in algorithmic structure, we present only the orthogonal operator quasi-Newton method for the Broyden family in Algorithm~\ref{al-3-4-2}.

\begin{algorithm}[!htbp]
	\caption{Orthogonal Operator Quasi-Newton Method for the Broyden Family}\label{al-3-4-2}
	\begin{algorithmic}[1]
		\State Set \( k := 0 \); given an initial point \( x_0 \in \mathbb{R}^n \), a symmetric positive definite matrix \( B_0 \), an integer \( 1 \le d \le n - 1 \), and precision \( \epsilon \geq 0 \).
		\State If \( \Vert g_k \Vert \leq \epsilon \), stop the algorithm.
		\State Update:
		\[
		x_{k+1} := x_k - \alpha_k B_k^{-1} g_k,
		\]
		where \( \alpha_k \) is determined by a line search.
		\State Compute \( s_k = x_{k+1} - x_k \), \( y_k = g_{k+1} - g_k \).
		\State Set \( \tilde{s}_k := s_k \), \( \tilde{y}_k := y_k \); then for \( j = \max\{0, k - d\}, \dots, k - 1 \), compute
		\[ 
		\alpha_{k,j} := \frac{\langle \tilde{s}_k,\, \tilde{y}_j \rangle}{\langle \tilde{s}_j,\, \tilde{y}_j \rangle},\] and update
		\[\tilde{s}_k := \tilde{s}_k - \alpha_{k,j} \tilde{s}_j, \quad
		\tilde{y}_k := \tilde{y}_k - \alpha_{k,j} \tilde{y}_j.
		\]
		\State Update \( B_{k+1} \) ensuring the Broyden family formula according to the quasi-Newton equation:
		\[
		B_{k+1} \tilde{s}_k = \tilde{y}_k.
		\]
		Set \( k := k + 1 \) and go back to step 2.
	\end{algorithmic}
\end{algorithm}

\begin{remark}
	Suppose that, in the \( k_0 \)-th iteration of Algorithm~\ref{al-3-4-2}, the condition \( \tilde{s}_{k_0}^\top \tilde{y}_{k_0} \ge 0 \) is not satisfied.  
	In this case, the method reverts to using the original vectors \( s_{k_0} \) and \( y_{k_0} \) for the update, and the orthogonalization procedure is restarted in subsequent iterations.  
	
	Specifically, set \( \tilde{s}_{k_0}=s_{k_0} \) and \( \tilde{y}_{k_0}=y_{k_0} \). For iterations \( k > k_0 \), orthogonalization is applied only to the sequence  
	\[
	\{\tilde{s}_j \mid j = \max\{k_0,\, k - d\}, \dots, k - 1\}.
	\]
That is, for each \( j = \max\{k_0,\, k - d\}, \dots, k - 1 \), compute
\[ 
\alpha_{k,j} := \frac{\langle \tilde{s}_k,\, \tilde{y}_j \rangle}{\langle \tilde{s}_j,\, \tilde{y}_j \rangle},\]
and then update
\begin{align*}
	\tilde{s}_k := \tilde{s}_k - \alpha_{k,j} \tilde{s}_j, \quad
	\tilde{y}_k := \tilde{y}_k - \alpha_{k,j} \tilde{y}_j.
\end{align*}
\end{remark}

In Algorithm~\ref{al-3-4-2}, the vector \(\tilde{y}_k\) aggregates secant information from iterations \(0\) to \(k\), leading to the accumulation of approximation errors over iterations, which can eventually compromise algorithm stability. To mitigate this issue and inspired by equation \eqref{eq2.7.3} in Proposition~\ref{p4}, we propose an improved approach that  project \( s_k \) directly onto the subspace 
\begin{equation}\label{insteadsubspace}
	\operatorname{span}\{s_{\max\{0, k - d\}}, \ldots, s_{k-1}\}.
\end{equation}

We solve the following weighted projection problem at each iteration:
\begin{equation}\label{Orthgonal equation}
	\min_{\beta} \| S_{k-1} \beta  - s_k \|_{W}^2 = \min_\beta (S_{k-1}\beta - s_k)^\top W (S_{k-1}\beta - s_k), \quad k \ge 1,
\end{equation}
where \[S_{k-1} = [s_{\max\{0, k - d\}}, \ldots, s_{k-1}]. \]
Denote by \(\beta_k\) the solution to \eqref{Orthgonal equation}, so that \(S_{k-1} \beta_k\) represents the projection of \(s_k\) onto the subspace \eqref{insteadsubspace} with respect to the \(W\)-inner product.
 Consequently, the \(W\)-orthogonal complement projection of \(s_k\), denoted \(\tilde{s}_k\), along with its associated secant vector \(\tilde{y}_k\), are computed as
\[
\tilde{s}_k = s_k - S_{k-1} \beta_k, \quad \tilde{y}_k = y_k - Y_{k-1} \beta_k,
\]
where \(Y_{k-1} = [y_{\max\{0, k - d\}}, \ldots, y_{k-1}]\).

Solving problem \eqref{Orthgonal equation} is equivalent to solving the normal equations:
\begin{equation}\label{the normal equations}
	S_{k-1}^\top W S_{k-1} \beta = S_{k-1}^\top W s_k.
\end{equation}
For the Broyden family, taking \(W = A\), equation \eqref{the normal equations} becomes
\begin{equation}\label{the normal equations of Broyden2}
	S_{k-1}^\top Y_{k-1} \beta = S_{k-1}^\top y_k.\nonumber
\end{equation}
Since \( S_{k-1}^\top Y_{k-1} \) is typically nonsymmetric in nonlinear contexts, a symmetrized system is  solved in practice:
\begin{equation}\label{the normal equations of Broyden}
	(S_{k-1}^\top Y_{k-1} + Y_{k-1}^\top S_{k-1}) \beta = S_{k-1}^\top y_k + Y_{k-1}^\top s_k.
\end{equation}
Equation \eqref{the normal equations of Broyden} corresponds to the optimality condition of the optimization problem
\begin{equation}\label{Orthgonal equation2}
	\min_\beta (S_{k-1} \beta - s_k)^\top (Y_{k-1} \beta - y_k).\nonumber
\end{equation}
For the generalized PSB family, taking \(W = M^{-2}\), equation \eqref{the normal equations} becomes
\begin{equation}\label{the normal equations of Broyden of psb}
	S_{k-1}^\top M^{-2} S_{k-1} \beta = S_{k-1}^\top M^{-2} s_k.\nonumber
\end{equation}
For BGM, taking \(W = I\), equation \eqref{the normal equations} becomes
\begin{equation}\label{the normal equations of Broyden of BGM}
	S_{k-1}^\top S_{k-1} \beta = S_{k-1}^\top s_k.\nonumber
\end{equation}

When the Hessian is ill-conditioned or the columns of \( S_{k-1} \) approach linear dependence, the system \eqref{the normal equations} may suffer from numerical instability. To alleviate this, it is standard practice to incorporate a regularization term \(\lambda I\) into the coefficient matrix, enhancing the robustness of the solution.


Owing to the similarity in algorithmic structure, we  present only the improved orthogonal operator quasi-Newton method for the Broyden family in Algorithm~\ref{al-3-4-3}.
\begin{algorithm}[!htbp]
	\caption{Improved Orthogonal Operator Quasi-Newton Method for Broyden Family}\label{al-3-4-3}
	\begin{algorithmic}[1]
		\State Set \(k := 0\); given an initial point \(x_0 \in \mathbb{R}^n\), a symmetric positive definite matrix \(B_0\), an integer \(1 \le d \le n-1\), and precision \(\epsilon \ge 0\).
		\State If \(\Vert g_k \Vert \le \epsilon\), stop the algorithm.
		\State Update:
		\[
		x_{k+1} := x_k - \alpha_k B_k^{-1} g_k,
		\]
		where \(\alpha_k\) is determined by a line search.
		\State Compute \(s_k := x_{k+1} - x_k\), \(y_k := g_{k+1} - g_k\).
		\State If \(k=0\), set \(\tilde{s}_k := s_k\), \(\tilde{y}_k := y_k\). Otherwise, solve the linear system \eqref{the normal equations of Broyden} to obtain \(\beta_k\), and set
		\[
		\tilde{s}_k := s_k-S_{k-1} \beta_k, \quad \tilde{y}_k := y_k-Y_{k-1} \beta_k.
		\]
		\State Update \(B_{k+1}\) using the Broyden family formula according to the quasi-Newton equation:
		\[
		B_{k+1} \tilde{s}_k = \tilde{y}_k.
		\]
		Store the vector sets
		\[
		S_k := [s_{\max\{0, k-d+1\}}, \ldots, s_k], \quad Y_k := [y_{\max\{0, k-d+1\}}, \ldots, y_k].
		\]
		Set \(k := k + 1\) and go back to step 2.
	\end{algorithmic}
\end{algorithm}

If the vector \( s_k \) is nearly linearly dependent on the columns of \( S_{k-1} \), then the projected vector 
\[
\tilde{s}_k = s_k - S_{k-1} \beta_k
\]
becomes very small in norm. This near-vanishing \(\tilde{s}_k\) amplifies the relative error in the secant information used for the matrix update, potentially degrading numerical stability and update quality. 
Therefore, when \(\|\tilde{s}_k\|\) falls below a certain threshold, it is advisable to discard the projected vectors and instead use the original pair \( (s_k, y_k) \) for the update.
\begin{remark}
	Algorithm~\ref{al-3-4-3} can be readily extended to L-BFGS. When L-BFGS stores \(N\) pairs of past curvature information, the number of vectors in \(S_k\), denoted by \(d\), should satisfy \(1 \leq d \leq N - 1\). In practice, a larger \( d \) means that the vectors \(\tilde{s}\) and \(\tilde{y}\) used in the matrix update incorporate more historical information. Moreover, increasing \( d \) also raises the risk of linear dependence among the vectors in \( S_k \), which in turn makes the subproblem \eqref{the normal equations of Broyden} more prone to ill-conditioning. In practice, \(d\) is generally chosen to be small.
\end{remark}
\section{Numerical experiment}\label{section5}

We tested the performance of the algorithms on the following quadratic function:
\begin{equation}\label{q2}
	f(x) = \frac{1}{2}\sum_{i=1}^{50} i\, x_i^2.
\end{equation}
The initial point is set to \(x_0 := [1, \dots, 1]^\top\). For different choices of the initial matrix \(B_0\), we record the number of iterations required by each method to meet the stopping criterion
\[
\|x_k - x_\star\|_2 \le 10^{-7} \, \|x_0 - x_\star\|_2,
\]
where \(x_\star\) denotes the exact solution. Specifically, we take
\[
B_0 := \lambda I, \quad \lambda \in \{50, 100, 200, 500, 1000, 5000\}.
\]

Table~\ref{tp13} compares the iteration counts with unit step length for the standard methods DFP, BFGS, PSB, L-BFGS and their corresponding image operator variants, as well as for the improved projection operator variants. For example, the image operator version of DFP is denoted by Im-DFP, and the improved projection operator version is denoted by IP-DFP. For L-BFGS, we tested the number of stored past curvature pairs \( N=3 \) and \( N=10 \). In tests of the improved projection operator methods, we use \( d=1 \) and \( d=2 \). In tests of the image operator methods, we use \( t_k=1 \).

\begin{table}[!htbp]
	\centering
	\small
	\setlength{\tabcolsep}{6pt}
	\renewcommand{\arraystretch}{1.2}
\begin{tabular}{lcccccc}
	\toprule
	$\lambda$ & 50 & 100 & 200 & 500 & 1000 & 5000 \\
	\midrule
	\textbf{DFP}              & 124  & 235  & 454  & 1121 & 2221 & 11096 \\
	Im-DFP           &  22  &  29  &  33  &   35 &   36 &    36 \\
	IP-DFP($d=1$)    & 100  & 181  & 345  &  830 & 1597 &  7480 \\
	IP-DFP($d=2$)    &  83  & 121  & 186  &  335 &  516 &  1711 \\
	\midrule
	\textbf{BFGS}             &  55  &  79  & 110  &  157 &  194 &   279 \\
	Im-BFGS          &  22  &  29  &  33  &   35 &   36 &    36 \\
	IP-BFGS($d=1$)   &  49  &  70  &  94  &  128 &  155 &   209 \\
	IP-BFGS($d=2$)   &  47  &  65  &  85  &  113 &  133 &   177 \\
	\midrule
	\textbf{PSB}              &  88  & 135  & 229  &  663 & 1554 &  9084 \\
	Im-PSB           &  21  &  29  &  33  &   35 &   36 &    36 \\
	IP-PSB($d=1$)    &  71  & 122  & 209  &  482 &  995 &  4273 \\
	IP-PSB($d=2$)    &  53  &  84  & 145  &  321 &  598 &  2717 \\
	\midrule
	\textbf{L-BFGS (N=3)}          & 120  & 173  & 203  &  478 &  862 &  3336 \\
	Im-LBFGS         &  32  &  41  &  37  &   40 &   37 &    36 \\
	IP-LBFGS($d=1$)  & 216  & 412  & 887  & 1337 & 5839 & 16142 \\
	IP-LBFGS($d=2$)  &  98  & 178  & 160  &  163 &  163 &   163 \\
	\midrule
\textbf{L-BFGS (N=10)}            &  81  & 128  & 223  &  240 &  313 &   453 \\
	Im-LBFGS         &  26  &  30  &  33  &   35 &   36 &    36 \\
	IP-LBFGS($d=1$)  &  71  & 115  & 125  &  208 &  239 &   412 \\
	IP-LBFGS($d=2$)  &  58  &  95  & 154  &  191 &  218 &   462 \\
	\bottomrule
\end{tabular}
	\caption{Comparisons between DFP, BFGS, PSB, LBFGS methods with their corresponding two operator methods on the number of iterations solving problems (\ref{q2})}	\label{tp13}
\end{table} 
Table \ref{tp13} shows that the use of the image operator significantly reduces the number of iterations, bringing the performance of several classical methods to a comparable level.  It is worth noting that in the numerical experiments with L-BFGS, even when storing only a few past curvature pairs (\(N = 3\)), the number of iterations remains low. This highlights the potential of image operator methods for large-scale problems.
Additionally, applying the projection operator improves the computational efficiency of DFP, BFGS, and PSB, with the case \( d=2 \) generally outperforming \( d=1 \). 
However, the behavior of the projection operator in L-BFGS is more complex. When L-BFGS utilizes fewer past curvature pairs, setting \( d=1 \) disrupts the original iteration structure and significantly increases the iteration count. Conversely, when more historical information is employed, the performance of the projection operator method tends to stabilize.

Table \ref{tab:lbfgs_smallN} indicates that, with fewer historical pairs in L-BFGS, using a larger \( d \) can lead to improved performance.

\begin{table}[!htbp]
	\centering
	\small
	\setlength{\tabcolsep}{8pt}
	\renewcommand{\arraystretch}{1.2}
	\begin{tabular}{lcccccc}
		\toprule
		$\lambda$ & 50 & 100 & 200 & 500 & 1000 & 5000 \\
		\midrule
		\textbf{L-BFGS (N=3)}         & 120 & 173 & 203 & 478 & 862 & 3336 \\
		IP-LBFGS ($d=1$)              & 216 & 412 & 887 & 1337 & 5839 & 16142 \\
		IP-LBFGS ($d=2$)              &  98 & 178 & 160 & 163 & 163 & 163 \\
		\midrule
		\textbf{L-BFGS (N=4)}         &  91 & 137 & 195 & 570 & 647 & 3426 \\
		IP-LBFGS ($d=1$)              & 123 & 186 & 298 & 225 & 329 & 969 \\
		IP-LBFGS ($d=2$)              &  68 &  97 & 111 & 129 & 135 & 145 \\
		IP-LBFGS ($d=3$)              &  87 &  84 &  82 &  87 & 129 & 175 \\
		\midrule
		\textbf{L-BFGS (N=5)}         &  94 & 146 & 226 & 279 & 304 & 979 \\
		IP-LBFGS ($d=1$)              &  95 & 134 & 239 & 406 & 704 & 3835 \\
		IP-LBFGS ($d=2$)              &  77 & 103 & 157 & 278 & 402 & 1157 \\
		IP-LBFGS ($d=3$)              &  61 &  70 &  81 & 103 & 135 & 249 \\
		\bottomrule
	\end{tabular}
	\caption{Iteration counts for L-BFGS with small memory size \(N\) and projection operator parameters \(d\).}
\label{tab:lbfgs_smallN}
\end{table}

Next, we present two numerical examples for solving nonlinear systems of equations. 
In the improved projection operator quasi-Newton method for BGM (IP-BGM), the parameter \( d \) is set to 1. 
We compare this method with Newton's method and the standard BGM. All methods use a unit step size.
\begin{example}
	Consider the nonlinear system
	\[
	\begin{cases}
		x_1^2 + x_2^2 = 1, \\
		x_1 = \cos(x_2)
	\end{cases}
	\]
	with the initial point \(x_0 := (0.5, 0.5)^\top\), and the initial matrix set to the identity. 
	The stopping criterion is
	\[
	\|F_k\|_2 \le 10^{-7}.
	\]
	The number of iterations required by each method is:
	\begin{itemize}
		\item Newton's method: 23 iterations,
		\item BGM: 572 iterations,
		\item IP-BGM(\(d=1\)): 51 iterations.
	\end{itemize}
\end{example}

\begin{example}
	Consider the following 10-dimensional nonlinear system:
	\[
	\forall i = 0, 2, 4, 6, 8:\quad
	\begin{cases}
		x_i^2 - x_{i+1} + 1 = 0, \\
		10(x_{i+1} - x_i^2) = 0
	\end{cases}
	\]
	which can be rearranged into the standard modified Rosenbrock form:
	\[
	\begin{cases}
		f_i(x) = 10(x_{i+1} - x_i^2), \\
		f_{i+1}(x) = 1 - x_i
	\end{cases}
	\quad \text{for } i = 0, 2, 4, 6, 8.
	\]
	
	The system has a known solution \(x^\star = (1, 1, \dots, 1)^\top \in \mathbb{R}^{10}\). 
	The initial point is set to \(x_0 := (0.5, 0.5, \dots, 0.5)^\top\), and the initial matrix is set to the identity. 
	The stopping criterion is
	\[
	\|F_k\|_2 \le 10^{-7}.
	\]
	The number of iterations required by each method is:
	\begin{itemize}
		\item Newton's method (with analytic Jacobian): 2 iterations,
		\item BGM: 15 iterations,
		\item IP-BGM(\(d=1\)): 5 iterations.
	\end{itemize}
\end{example}

\section{Conclusion}\label{c:6}
This paper presents a unified quadratic termination framework from the perspective of matrix approximation, systematically analyzing improved matrix approximation of quasi-Newton methods. Based on this, image operators and projection operators are introduced to effectively map correction directions onto specific subspaces, thereby optimizing the traditional quasi-Newton update matrices.   Practical algorithms for DFP, BFGS, PSB, L-BFGS, and BGM are designed within this operator framework. Preliminary numerical experiments demonstrate that these two types of operator methods significantly accelerate convergence, verifying the potential and effectiveness of the operator perspective in enhancing quasi-Newton performance.

Chapter~\ref{section2} introduces a unified quadratic termination framework covering the Broyden family, the generalized PSB family, and the Broyden's ``good''  method. By defining a suitable inner product, the correction direction is mapped to the orthogonal complement of the kernel space. This approach achieves quadratic termination without requiring exact line searches. Based on an equivalent characterization of the orthogonal complement of the kernel space, image operators are proposed. Subsequently, we derive a relaxed quadratic termination condition and propose a projection operator approach employing a stepwise orthogonalization strategy.

Chapter~\ref{section3} proves the advantages of image and projection operators over standard methods in matrix approximation. Starting from the generalized PSB family and its dual family, the results are then extended to DFP, BFGS, and BGM. Additionally, we analyze the improved matrix approximation properties obtained by projecting \( s_k \) onto the orthogonal complement of a subspace of \( \ker(B_k - A) \), providing theoretical support for the stepwise orthogonalization approach.

Chapter~\ref{section4} designs practical algorithms based on these operator approaches, including approximate computations of operators for the Broyden family and the generalized PSB family. It also introduces an improved projection technique that directly projects onto the subspace \(\operatorname{span}\{s_{k-d}, \ldots, s_{k-1}\}\) for mitigating error accumulation and improving the stability of the update. Furthermore, the chapter points out that both operator methods can be naturally extended to L-BFGS.

Chapter~\ref{section5} presents numerical experiments validating that operator-based methods significantly enhance various quasi-Newton algorithms. In particular, the use of projection operators greatly improves BGM, while the use of image operators substantially reduces iteration counts for DFP, BFGS, PSB, and L-BFGS. 

While image operator methods can significantly reduce the number of iterations, they require additional gradient evaluations at each iteration. Therefore, in practice, a trade-off between gradient computation cost and convergence rate must be considered. The projection operator does not require extra gradient evaluations, but its robustness may be affected by the linear dependence of search directions. As a remedy, one can consider incorporating linear-dependence detection and dynamically adjust the number of stored past curvature pairs used in the projection.

Preliminary numerical experiments demonstrate that operator-based methods can effectively enhance the performance of quasi-Newton schemes, showing promising potential for solving large-scale optimization problems and large-scale systems of nonlinear equations. In subsequent versions, we intend to explore the applicability conditions and computational optimizations of the proposed operator methods, and to perform evaluations using the CUTEst benchmark collection.

\bibliographystyle{plain}
\bibliography{ref}
\appendix
\end{document}